\documentclass[11pt]{article}
\usepackage[margin = 1.2in]{geometry}
\usepackage{dsfont}
\usepackage{float}
\usepackage{hhline}
\usepackage{amsmath,amssymb,amsfonts,amsthm}
\usepackage{url,graphicx,tabularx,array,appendix,geometry,multirow,array}
\usepackage{xy}\xyoption{all} \xyoption{poly} \xyoption{knot}
\usepackage{titlesec}
\usepackage{natbib}
\usepackage[usenames,dvipsnames]{color}
\usepackage{enumerate}
\usepackage{latexsym}
\usepackage{subfigure}

\titleformat{\section}{\sf\large}{\thesection.}{1em}{\filcenter}
\titleformat{\subsection}{\sf}{\thesubsection}{1em}{}
\titleformat{\subsubsection}{\sf}{\thesubsubsection}{1em}{}
\newcolumntype{L}[1]{>{\raggedright\let\newline\\\arraybackslash\hspace{0pt}}m{#1}}
\newcolumntype{C}[1]{>{\centering\let\newline\\\arraybackslash\hspace{0pt}}m{#1}}
\newcolumntype{R}[1]{>{\raggedleft\let\newline\\\arraybackslash\hspace{0pt}}m{#1}}

\numberwithin{equation}{section}

\def\bbR{\mathbb{R}}

\newtheorem{theorem}{Theorem}[section]
\newtheorem{lemma}[theorem]{Lemma}
\newtheorem{corollary}[theorem]{Corollary}

\makeatletter\def\theequation{\arabic{section}.\arabic{equation}}
\@addtoreset{equation}{section}\makeatother

\title{Minimax Optimal Bayesian Aggregation}
\author{Yun Yang and David Dunson}
\date{}

\begin{document}
\bibliographystyle{chicago}
\maketitle

\begin{abstract}
It is generally believed that ensemble approaches, which combine multiple algorithms or models, can outperform any single algorithm at machine learning tasks, such as prediction.  In this paper, we propose Bayesian convex and linear aggregation approaches motivated by regression applications.  We show that the proposed approach is minimax optimal when the true data-generating model is a convex or linear combination of models in the list.  Moreover, the method can adapt to sparsity structure in which certain models should receive zero weights, and the method is tuning parameter free unlike competitors.  More generally, under an M-open view when the truth falls
outside the space of all convex/linear combinations, our theory suggests that the posterior measure tends to concentrate on the best approximation of the truth at the minimax rate.  We illustrate the method through simulation studies and several applications.

\end{abstract}

\section{Introduction}
In many applications, it is not at all clear how to pick one most suitable method out of a list of possible models or learning algorithms $\mathcal{M} = \{ \mathcal{M}_1,\ldots,\mathcal{M}_M \}$.  Each model/algorithm has its own set of implicit or explicit assumptions under which that approach will obtain at or near optimal performance.  However, in practice verifying which if any of these assumptions hold for a real application is problematic.  Hence, it is of substantial practical importance to have an aggregating mechanism that can automatically combine the estimators $\hat{f}_1,\ldots, \hat{f}_M$ obtained from the $M$ different approaches $\mathcal{M}_1,\ldots,\mathcal{M}_M$, with the aggregated estimator potentially better than any single one.

Towards this goal, three main aggregation strategies receive most attention in the literature: model selection aggregation (MSA), convex aggregation (CA) and linear aggregation (LA), as first stated by \citet{Nemirovski2000}. MSA aims at selecting the optimal single estimator from the list; CA considers searching for the optimal convex combination of the estimators; and LA focuses on selecting the optimal linear combination. Although there is an extensive literature \citep{Juditsky2000,Tsybakov2003,Wegkamp2003,Yuhong2000,Yuhong2001,Yuhong2004,
Bunea2008,Bunea2007,Guedj2013,vanderlann2007} on aggregation, there has been limited consideration of Bayesian approaches.

In this paper, we study Bayesian aggregation procedures and their performance in regression. Consider the regression model
\begin{align}\label{eq:truth}
    Y_i=f(X_i)+\epsilon_i,\quad i=1,\ldots,n,
\end{align}
where $Y_i$ is the response variable, $f:\mathcal{X}\to\bbR$ is an unknown regression function, $\mathcal{X}$ is the feature space, $X_i$'s are the fixed- or random-designed elements in $\mathcal{X}$ and the errors are iid Gaussian.

Aggregation procedures typically start with randomly dividing the sample $D_n=\{(X_1,Y_1),$ $\ldots,$ $(X_n,Y_n)\}$ into a training set for constructing estimators $\hat{f}_1,\ldots,\hat{f}_M$, and a learning set for constructing
$\hat{f}$. Our primary interest is in the aggregation step, so we adopt the convention \citep{Bunea2007} of fixing the training set and treating the estimators $\hat{f}_1,\ldots,\hat{f}_M$ as fixed functions $f_1,\ldots,f_M$. Our results can also be translated to the context where the fixed functions $f_1,\ldots,f_M$ are considered as a functional basis \citep{Juditsky2000}, either orthonormal or overcomplete, or as ``weak learners" \citep{vanderlann2007}. For example, high-dimensional linear regression is a special case of LA where $f_j$ maps an $M$-dimensional vector into its $j$th component.

Bayesian model averaging (BMA) \citep{Hoeting1999} provides an approach for aggregation, placing a prior over the ensemble and then updating using available data to obtain posterior model probabilities. For BMA, $\hat{f}$ can be constructed as a convex combination of estimates $\hat{f}_1,\ldots,\hat{f}_M$ obtained under each model, with weights corresponding to the posterior model probabilities.  If the true data generating model $f_0$ is one of the models in the pre-specified list (``$\mathcal{M}$-closed" view), then as the sample size increases the weight on $f_0$ will typically converge to one.
With a uniform prior over $\mathcal{M}$ in the regression setting with Gaussian noise, $\hat{f}$ coincides with the exponentially weighted aggregates \citep{Tsybakov2003}. However, BMA relies on the assumption that $\mathcal{M}$ contains the true model. If this assumption is violated (``$\mathcal{M}$-open"), then $\hat{f}$ tends to converge to the single model in $\mathcal{M}$ that is closest to the true model in Kullback-Leibler (KL) divergence. For example, when $f_0$ is a weighted average of $f_1$ and $f_2$, under our regression setting $\hat{f}$ will converge to $f\in\{f_1,f_2\}$ that minimizes $||f-f_0||_n^2=n^{-1}\sum_{i=1}^n|f(X_i)-f_0(X_i)|^2$ under fixed design or $||f-f_0||^2_Q=E_Q|f(X)-f_0(X)|^2$ under random design where $X\sim Q$. Henceforth, we use the notation $||\cdot||$ to denote $||\cdot||_n$ or $||\cdot||_Q$ depending on the context.

In this paper, we primarily focus on Bayesian procedures for CA and LA. Let
$$
\mathcal{F}^{H}=\big\{f_{\lambda}=\sum_{j=1}^M\lambda_jf_{j}: \lambda=(\lambda_1,\ldots,\lambda_M)\in H\big\}
$$
be the space of all aggregated estimators for $f_0$ with index set $H$. For CA, $H$ takes the form of $\Lambda=\{(\lambda_1,\ldots,\lambda_M):\lambda_j\geq0, j=1,\ldots,M, \ \sum_{j=1}^M\lambda_j=1\}$ and for LA, $H=\Omega=\{(\lambda_1,\ldots,\lambda_M):\lambda_j\in\bbR,j=1,\ldots,M, \ \sum_{j=1}^M|\lambda_j|\leq L\}$, where $L>0$ can be unknown but is finite. In addition, for both CA and LA we consider sparse aggregation with $\mathcal{F}^{H_s}$, where an extra sparsity structure $||\lambda||_0=s$ is imposed on the weight $\lambda\in H_s=\{\lambda\in H:||\lambda||_0=s\}$. Here, for a vector $\theta\in\bbR^M$, we use $||\theta||_p=(\sum_{j=1}^M|\theta_j|^p)^{1/p}$ to denotes its $l_p$-norm for $0\leq p\leq \infty$. In particular, $||\theta||_0$ is the number of nonzero components of $\theta$. The sparsity level $s$ is allowed to be unknown and expected to be learned from data. In the sequel, we use the notation $f_{\lambda^\ast}$ to denote the best $||\cdot||$-approximation of $f_0$ in $\mathcal{F}^{H}$. Note that if $f_0\in\mathcal{F}^{H}$, then $f_0=f_{\lambda^\ast}$.

One primary contribution of this work is to propose a new class of priors, called Dirichlet aggregation (DA) priors, for Bayesian aggregation. Bayesian approaches with DA priors are shown to lead to the minimax optimal posterior convergence rate over $\mathcal{F}^{H}$ for CA and LA, respectively. More interestingly, DA is able to achieve the minimax rate of sparse aggregation (see Section \ref{se:minimax}), which improves the minimax rate of aggregation by utilizing the extra sparsity structure on $\lambda^\ast$.
This suggests that DA is able to automatically adapt to the unknown sparsity structure when it exists but also has optimal performance in the absence of sparsity. Such sparsity adaptive properties have also been observed in \cite{Bunea2007} for penalized optimization methods. However, in order to achieve minimax optimality, the penalty term, which depends on either the true sparsity level $s$ or a function of $\lambda^\ast$, needs to be tuned properly. In contrast, the DA does not require any prior knowledge on $\lambda^\ast$ and is tuning free.

Secondly, we also consider an ``M-open" view for CA and LA, where the truth $f_0$ can not only fall outside the list $\mathcal{M}$, but also outside the space of all convex/linear combinations of the models in $\mathcal{M}$. Under the ``M-open" view, our theory suggests that the posterior measure tends to put all its mass into a ball around the best approximation $f_{\lambda^\ast}$ of $f_0$ with a radius proportional to the minimax rate. The metric that defines that ball will be made clear later. This is practically important because the true model in reality is seldom correctly specified and a convergence to $f_{\lambda^\ast}$ is the best one can hope for. Bayesian asymptotic theory for misspecified models is under developed, with most existing results assuming that the model class is either known or is an element of a known list. One key step is to construct appropriate statistical tests discriminating $f_{\lambda^*}$ from other elements in $\mathcal{F}^{H}$. Our tests borrow some results from \cite{Kleijn2006} and rely on concentration inequalities.

The proposed prior on $\lambda$ induces a novel shrinkage structure, which is of independent interest.  There is a rich literature on theoretically optimal models based on discrete (point mass mixture) priors \citep{Ishwaran2005,Castillo2012} that are supported on a combinatorial model space, leading to heavy computational burden. However, continuous shrinkage priors avoid stochastic search variable selection algorithms \citep{George1997} to sample from the combinatorial model space and can potentially improve computational efficiency. Furthermore, our results include a rigorous investigation on $M$-dimensional symmetric Dirichlet distributions, Diri$(\rho,\ldots,\rho)$ when $M\gg 1$ and $\rho\ll1$. Here Diri$(\alpha_1,\ldots,\alpha_M)$ denotes a Dirichlet distribution with concentration parameters $\alpha_1,\ldots,\alpha_M$. In machine learning, Diri$(\rho,\ldots,\rho)$ with $\rho\ll 1$ are widely used as priors for latent class probabilities \citep{Blei2003}. However, little rigorous theory has been developed for the relationship between its concentration property and the hyperparameter $\rho$.
\citet{Rousseau2011} consider a related problem of overfitted mixture models and show that generally the posterior distribution effectively empties the extra components. However, our emphasis is to study the prediction performance instead of model selection. Moreover, in \cite{Rousseau2011} the number $M$ of components is assumed to be fixed as $n$ increases, while in our setting we allow $M$ to grow in the order of $e^{o(n)}$. In this large-$M$ situation, the general prior considered in \cite{Rousseau2011} is unable to empty the extra components and we need to impose sparsity.
In this paper, we show that if we choose $\rho\sim M^{-\gamma}$ with $\gamma>1$, then Diri$(\rho,\ldots,\rho)$ could lead to the optimal concentration rate for sparse weights (Section \ref{se:cdd}). Moreover, such concentration is shown to be adaptive to the sparsity level $s$.

The rest of the paper is organized as follows. In Section 1.1, we review the minimax results for aggregation. In Section 2, we describe the new class of priors for CA and LA based on symmetric Dirichlet distributions. In Section 3, we study the asymptotic properties of the proposed Bayesian methods. In Section 4, we show some simulations and applications. The proofs of the main theorems appear in Section 5 and some technical proofs are deferred to Section 6. We provide details of the MCMC implementation of our Bayesian aggregation methods in the appendix.

\subsection{A brief review of the Minimax risks for aggregation}\label{se:minimax}
It is known \citep{Tsybakov2003} that for CA, the minimax risk for estimating the best convex combination $f_{\lambda^\ast}$ within $\mathcal{F}^{\Lambda}$ is
\begin{align}\label{eq:CAmini}
    \sup_{f_1,\ldots,f_M\in\mathcal{F}_0}\inf_{\hat{f}}\sup_{f_{\lambda}^\ast\in\mathcal{F}^{\Lambda}}
    E||\hat{f}-f_{\lambda}^\ast||^2\asymp \left\{
                \begin{array}{cl}
                  M/n, & \text{if } M\leq\sqrt{n}, \\
                  \sqrt{\frac{1}{n}\log\big(M/\sqrt{n}+1\big)}, & \text{if } M>\sqrt{n},
                \end{array}
              \right.
\end{align}
where $\mathcal{F}_0=\big\{f:||f||_{\infty}\leq 1\}$ and $\hat{f}$ ranges over all possible estimators based on $n$ observations.
Here, for any two positive sequences $\{a_n\}$ and $\{b_n\}$, $a_n\asymp b_n$ means that there exists a constant $C>0$, such that $a_n\leq C b_n$ and $b_n\leq C a_n$ for any $n$. The norm is understood as the $L_2$-norm for random design and the $||\cdot||_n$-norm for fixed design.
If we have more information that the truth $f_{\lambda}^\ast$ also possesses a sparse structure $||\lambda^\ast||_0\triangleq \#\{j:\lambda_j>0\}=s\ll n$, then we would expect a faster convergence rate of estimating $f_{\lambda}^\ast$. For example, in the ``M-closed" case where $f_{\lambda}^\ast=f_j$ for some $j\in\{1,\ldots,M\}$, $\lambda_i^\ast=I(i=j)$ and $||\lambda^\ast||_0=1$. Let
$
\mathcal{F}^{\Lambda}_s=\big\{f=\sum_{j=1}^M\lambda_jf_{j}: \lambda\in\Lambda, ||\lambda||_0=s\big\}
$
be the space of all $s$-sparse convex aggregations of $f_1,\ldots,f_M$. By extending the results in \cite{Tsybakov2003}, it can be shown that when the sparsity level $s$ satisfies
$s\leq\sqrt{n/\log M}$, the minimax risk of estimating an element in $\mathcal{F}^{\Lambda}_s$ is given by
\begin{align}\label{eq:SCAmini}
    \sup_{f_1,\ldots,f_M\in\mathcal{F}_0}\inf_{\hat{f}}\sup_{f_{\lambda}^\ast \in\mathcal{F}^{\Lambda}_s}
    E||\hat{f}-f_{\lambda}^\ast ||^2\asymp \frac{s}{n}\log\bigg(\frac{M}{s}\bigg).
\end{align}
From the preceding results, $\sqrt{n/\log M}$ serves as the sparsity/non-spasrsity boundary of the weight $\lambda^\ast$ as there is no gain in the estimation efficiency if $s>\sqrt{n/\log M}$.

From \cite{Tsybakov2003}, the minimax risk for LA with $H=\bbR^M$ is
\begin{align*}
    \sup_{f_1,\ldots,f_M\in\mathcal{F}_0}\inf_{\hat{f}}\sup_{f_{\lambda}^\ast\in\mathcal{F}^{\bbR^M}}
    E||\hat{f}-f_{\lambda}^\ast||^2\asymp M/n.
\end{align*}
As a result, general LA is only meaningful when $M/n\to 0$, as $n\to\infty$.
Similarly, the above minimax risk can be extended to $s$-sparse LA $\mathcal{F}^{\bbR^M}_s=\big\{f=\sum_{j=1}^M\lambda_jf_{j}: \lambda\in\bbR^M, ||\lambda||_0=s\big\}$ for $s\in\{1,\ldots,M\}$ as
\begin{align*}
    \sup_{f_1,\ldots,f_M\in\mathcal{F}_0}\inf_{\hat{f}}\sup_{f_{\lambda}^\ast\in\mathcal{F}^{\bbR^M}_s}
    E||\hat{f}-f_{\lambda}^\ast||^2\asymp \frac{s}{n}\log\bigg(\frac{M}{s}\bigg).
\end{align*}
Note that for sparse LA, the sparsity level $s$ can be arbitrary. A simple explanation is that the constraint $||\lambda^\ast||_1=1$ ensures that every element in $\mathcal{F}^{\Lambda}$ can be approximated with error at most $\sqrt{\frac{1}{n}\log\big(M/\sqrt{n}+1\big)}$ by some $\sqrt{n/\log M}$-sparse element in $\mathcal{F}^{\Lambda}$ (see Lemma \ref{le:prep2}).
However, if we further assume that $||\lambda^*||\leq A$ and restrict $f^{\lambda^\ast}\in \mathcal{F}^{\Omega}$, then by extending \cite{Tsybakov2003}, it can be shown that the minimax risks of LA of $\mathcal{F}^{\bbR^M_A}$ is the same as those of convex aggregation under a non-sparse structure as \eqref{eq:CAmini} and a sparse structure as \eqref{eq:SCAmini}.

\section{Bayesian approaches for aggregation}

\subsection{Concentration properties of high dimensional symmetric Dirichlet distributions}\label{se:cdd}
Consider an $M$-dimensional symmetric Dirichlet distribution Diri$(\rho,\ldots,\rho)$ indexed by a concentration parameter $\rho>0$, whose pdf at $\lambda\in\Lambda$ is given by
$\Gamma(M\rho)\{\Gamma(\rho)\}^{-M}\prod_{j=1}^M\lambda_j^{\rho-1}$,
where $\Gamma(\cdot)$ is the Gamma function. $M$-dimensional Dirichlet distributions are commonly used in Bayesian procedures as priors over the $M-1$-simplex. For example, Dirichlet distributions can be used as priors for probability vectors for latent class allocation. In this subsection, we investigate the concentration properties of Diri$(\rho,\ldots,\rho)$ when $M\gg1$ and $\rho\ll1$. Fig.~\ref{fig:1} displays typical patterns for $3$-dimensional Dirichlet distributions Diri$(\rho,\rho,\rho)$ with $\rho$ changing from moderate to small. As can be seen, the Dirichlet distribution tends to concentrate on the boundaries for small $\rho$, which is suitable for capturing sparsity structures.

\begin{figure}[htp]
\centering
\subfigure[$\rho=1$.]{
   \includegraphics[trim=4cm 10cm 4cm 10cm, clip=true,width=1.8in]{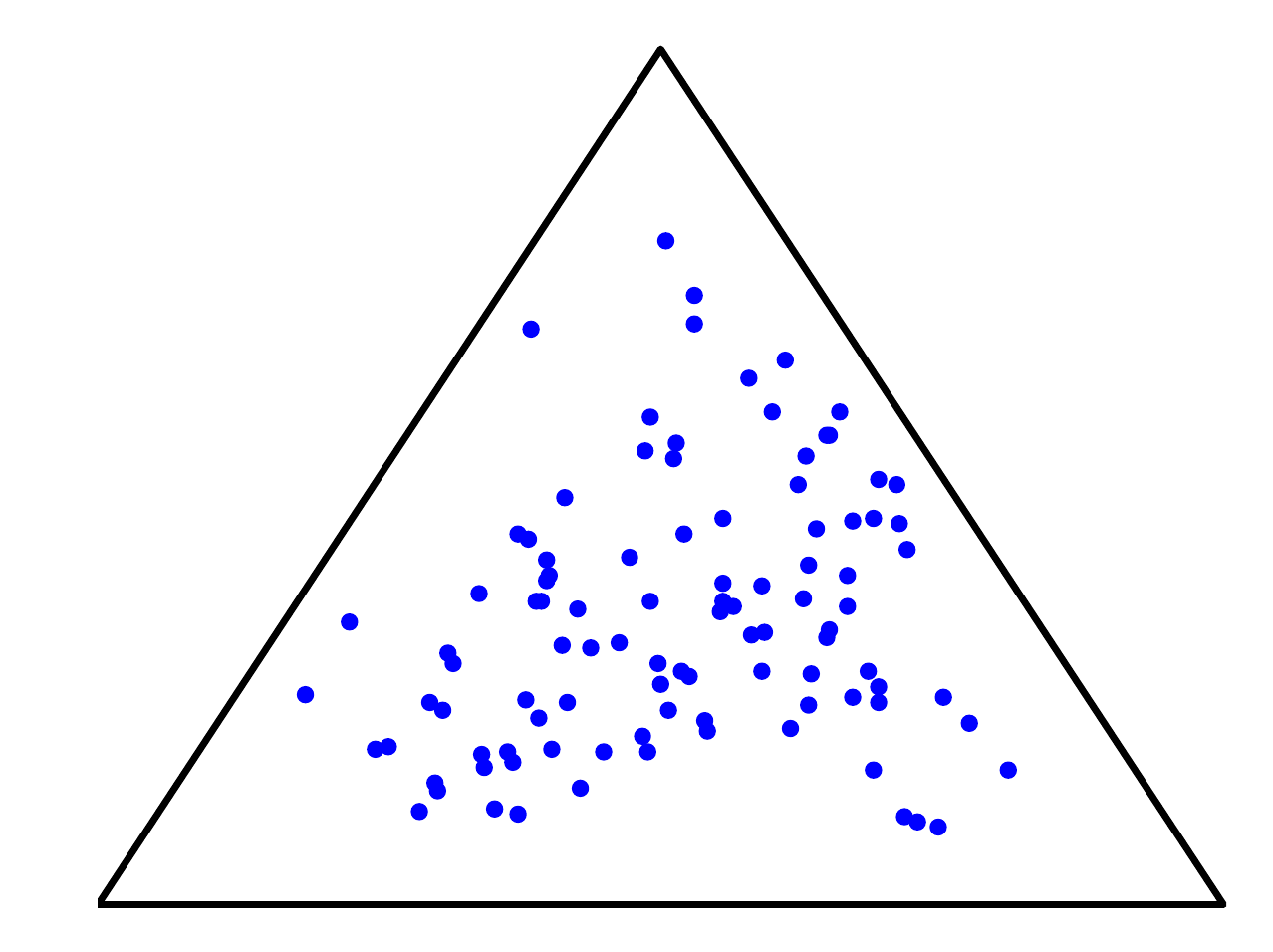}\vspace{-0.5cm}
\label{fig:f1}
 }
 \subfigure[$\rho=0.1$.]{
   \includegraphics[trim=4cm 10cm 4cm 10cm, clip=true,width=1.8in]{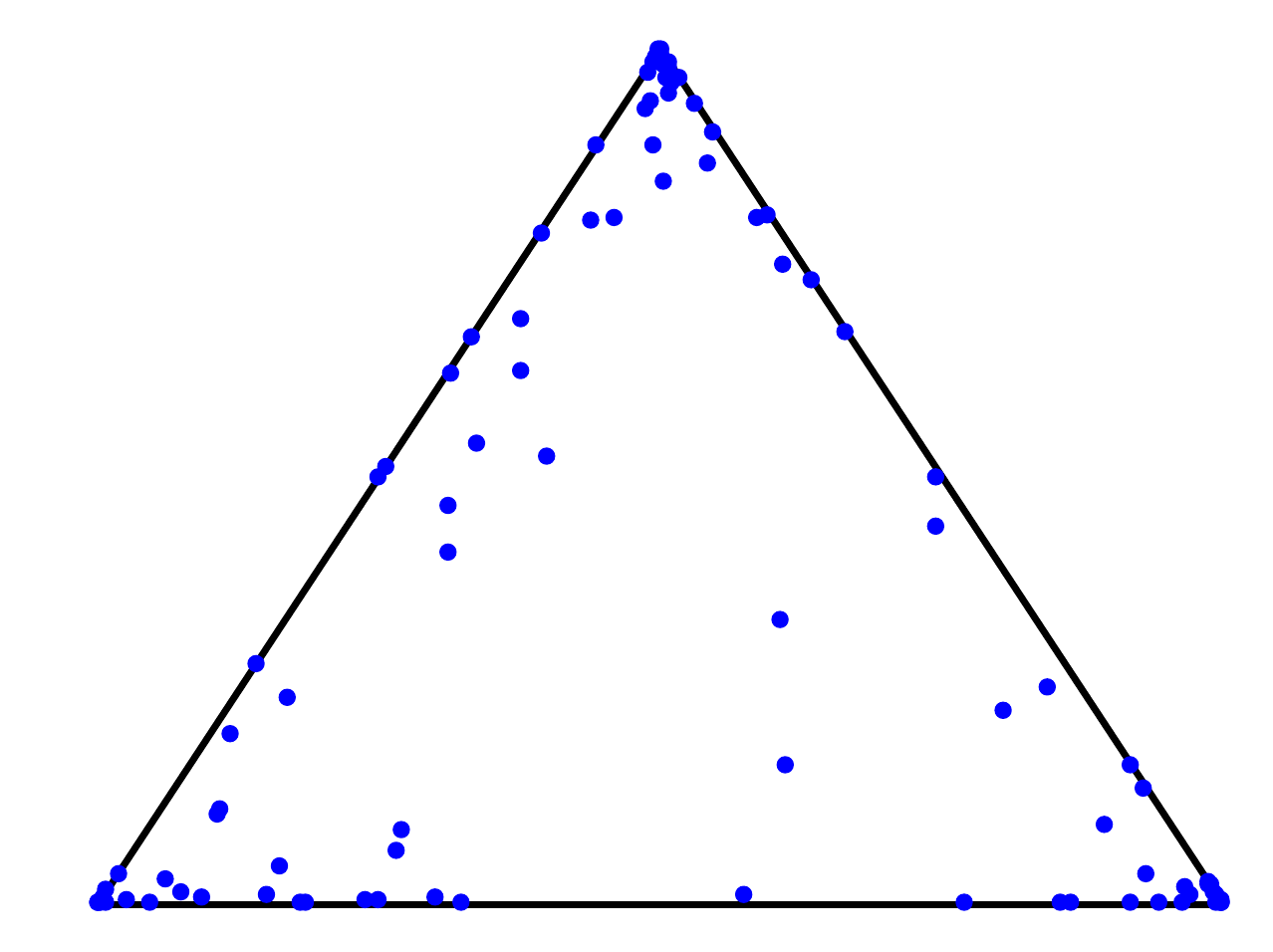}\vspace{-0.5cm}
\label{fig:f2}
 }
 \subfigure[$\rho=0.01$.]{
   \includegraphics[trim=4cm 10cm 4cm 10cm, clip=true,width=1.8in]{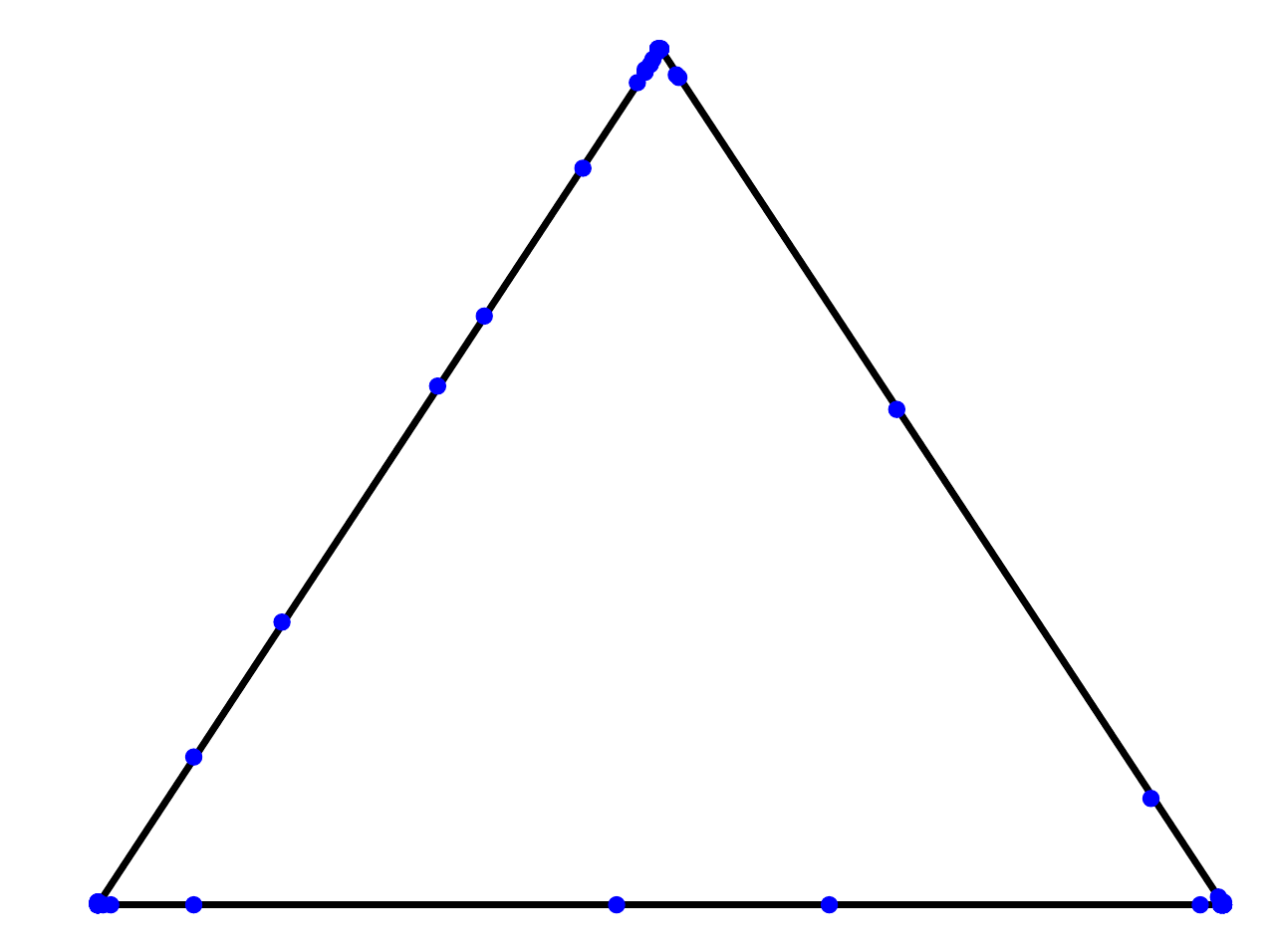}\vspace{-0.5cm}
\label{fig:f3}
 }
\label{fig:1}
\caption{Symmetric Dirichlet distributions with different values for the concentration parameter. Each plot displays 100 independent draws from Diri$(\rho,\rho,\rho)$.}
\end{figure}

To study the concentration of Diri$(\rho,\ldots,\rho)$, we need to characterize  the space of sparse weight vectors. Since Dirichlet distributions are absolutely continuous, the probability of generating an exactly $s$-sparse vector is zero for any $s<M$.
Therefore, we need to relax the definition of $s$-sparsity. Consider the following set indexed by a tolerance level $\epsilon>0$ and a sparsity level $s\in\{1,\ldots,M\}$: $\mathcal{F}^{\Lambda}_{s,\epsilon}=\{\lambda\in \Lambda: \sum_{j=s+1}^{M}\lambda_{(j)}\leq \epsilon\}$, where $\lambda_{(1)}\geq\lambda_{(2)}\geq\cdots\geq\lambda_{(M)}$ is the ordered sequence of $\lambda_1,\ldots,\lambda_M$. $\mathcal{F}^{\Lambda}_{s,\epsilon}$ consists of all vectors that can be approximated by $s$-sparse vectors with $l_1$-error at most $\epsilon$. The following theorem shows the concentration property of the symmetric Dirichlet distribution Diri$(\rho,\ldots,\rho)$ with $\rho=\alpha/M^{\gamma}$. This theorem is a easy consequence of Lemma \ref{le:conprob} and Lemma \ref{le:cpp} in Section \ref{se:proofs}.

\begin{theorem}\label{thm:dd}
Assume that $\lambda\sim$ Diri$(\rho,\ldots,\rho)$ with $\rho=\alpha/M^{\gamma}$ and $\gamma> 1$. Let $\lambda^\ast\in\Lambda_s$ be any $s$-sparse vector in the $M-1$-dimensional
simplex $\Lambda$. Then for any $\epsilon\in(0,1)$ and some $C>0$,
\begin{align}
   & P(||\lambda-\lambda^\ast||_2\leq \epsilon)\gtrsim \exp\bigg\{-C\gamma s\log\frac{M}{\epsilon} \bigg\},\label{eq:sbpCA}\\
   &P(\lambda\notin \mathcal{F}^{\Lambda}_{s,\epsilon})\lesssim \exp\bigg\{-C(\gamma-1)s\log \frac{M}{\epsilon}\bigg\}.
    \label{eq:suppCA}
\end{align}
\end{theorem}

The proof of \eqref{eq:suppCA} utilizes the stick-breaking representation of Dirichlet processes \citep{sethuraman1994} and the fact that Diri$(\rho,\ldots,\rho)$ can be viewed as the joint distribution of $\big(G([0,1/M)),$ $\ldots,$ $G([(M-1)/M,1))\big)$ where $G\sim$ Dirichlet process DP$((M\rho) U)$ with $U$ the uniform distribution on $[0,1]$. The condition $\gamma>1$ in Theorem \ref{thm:dd} reflects the fact that the concentration parameter $M\rho=\alpha M^{-(\gamma-1)}$ should decrease to $0$ as $M\to\infty$ in order for DP$((M\rho) U)$ to favor sparsity.
\eqref{eq:suppCA} validates our observations in Fig.~\ref{fig:1} and \eqref{eq:sbpCA} suggests that the prior mass around every sparse vector is uniformly large since the total number of $s$-sparse patterns (locations of nonzero components) in $\Lambda$ is of order $\exp\{Cs\log(M/s)\}$. In fact, both \eqref{eq:sbpCA} and \eqref{eq:suppCA} play crucial roles in the proofs in Section \ref{se:pcr} on characterizing the posterior convergence rate $\epsilon_n$ for the Bayesian method below for CA (also true for more general Bayesian methods), where $\{\epsilon_n\}$ is a sequence satisfying $P(||\lambda-\lambda^\ast||_2\leq\epsilon_n)\gtrsim \exp(-n\epsilon_n^2)$  and $P(\lambda\notin \mathcal{F}^{\Lambda}_{s,\epsilon})\lesssim \exp(-n\epsilon_n^2)$. Assume the best approximation $f_{\lambda^\ast}$ of the truth $f_0$ to be $s$-sparse. \eqref{eq:suppCA} implies that the posterior distribution of $\lambda$ tends to put almost all its mass in $\mathcal{F}^{\Lambda}_{s,\epsilon}$ and \eqref{eq:sbpCA} is required for the posterior distribution to be able to concentrate around $\lambda^\ast$ at the desired minimax rate given by \eqref{eq:CAmini}.

\subsection{Using Dirichlet priors for Convex Aggregation}\label{se:BA}
In this subsection, we assume $X_i$ to be random with distribution $Q$ and $f_0\in L_2(Q)$. Here, for a probability measure $Q$ on a space $\mathcal{X}$, we use the notation $||\cdot||_Q$ to denote the norm associated with the square integrable function space $L_2(Q)=\{f:\int_{\mathcal{X}}|f(x)|^2dQ(x)\leq\infty\}$. We assume the random design for theoretical convenience and the procedure and theory for CA can also be generalized to fixed design problems. Assume the $M$ functions $f_1,\ldots,f_M$ also belong to $L_2(Q)$.
Consider combining these $M$ functions into an aggregated estimator $\hat{f}=\sum_{j=1}^M\hat{\lambda}_jf_{j}$, which tries to estimate $f_0$ by elements in the space
$\mathcal{F}^{\Lambda}=\big\{f=\sum_{j=1}^M\lambda_jf_{j}: \lambda_j\geq0, \sum_{j=1}^M\lambda_j= 1\big\}$
of all convex combinations of $f_1,\ldots,f_M$. The assumption that $f_1,\ldots,f_M$ are fixed is reasonable as long as different subsets of samples are used for producing $f_1,\ldots,f_M$ and for aggregation. For example, we can divide the data into two parts and use the first part for estimating $f_1,\ldots,f_M$ and the second part for aggregation.

We propose the following Dirichlet aggregation (DA) prior:
\begin{align*}
  &\text{(DA)}& f=\sum_{j=1}^M\lambda_jf_j,\ (\lambda_1,\ldots,\lambda_M)\sim \text{Diri}\bigg(\frac{\alpha}{M^{\gamma}},\ldots,\frac{\alpha}{M^{\gamma}}\bigg),&&
\end{align*}
where $(\gamma,\alpha)$ are two positive hyperparameters.
As Theorem \ref{thm:dd} and the results in Section \ref{se:proofs} suggest, such a symmetric Dirichlet distribution is favorable since Diri$(\alpha_1,\ldots,\alpha_M)$ with equally small parameters $\alpha_1=\ldots=\alpha_M=\alpha/M^{\gamma}$ for $\gamma>1$ has nice concentration properties under both sparse and nonsparse $L_1$ type conditions, leading to near minimax optimal posterior contraction rate under both scenarios.

We also mention a related paper \citep{Anirban2013} that uses Dirichlet distributions in high dimensional shrinkage priors, where they considered normal mean estimating problems. They proposed a new class of Dirichlet Laplace priors for sparse problems, with the Dirichlet placed on scaling parameters of Laplace priors for the normal means.  Our prior is fundamentally different in using the Dirichlet directly for the weights $\lambda$, including a power $\gamma$ for $M$.  This is natural for aggregation problems, and we show that the proposed prior is simultaneously minimax optimal under both sparse and nonsparse conditions on the weight vector $\lambda$ as long as $\gamma>1$.

\subsection{Using Dirichlet priors for Linear Aggregation}\label{se:lreg}
For LA, we consider a fixed design for $X_i\in\bbR^d$ and write \eqref{eq:truth} into vector form as $Y=F_0+\epsilon$, $\epsilon\sim N(0,\sigma^2I_n)$, where $Y=(Y_1,\ldots,y_n)$ is the $n\times1$ response vector,  $F_0=(f_0(X_1),\ldots,f_0(X_n))^T$ is the $n\times1$ vector representing the expectation of $Y$ and $I_n$ is the $n\times n$ identity matrix.
Let $F=(F_{ij})=(f_j(X_x))$ be the $n\times M$ prediction matrix, where the $j$th column of $F$ consists of all values of $f_j$ evaluated at the training predictors $X_1,\ldots,X_n$. LA estimates $F_0$ as $F\lambda$ with $\lambda=(\lambda_1,\ldots,\lambda_M)^T\in\bbR^M$ the $p\times 1$ the coefficient vector. Use the notation $F_j$ to denote the $j$th column of $F$ and $F^{(i)}$ the $i$th row. Notice that this framework of linear aggregation includes (high-dimensional) linear models as a special case where $d=M$ and $f_j(X_i) = X_{ij}$.

Let $A=||\lambda||_1=\sum_{j=1}^M|\lambda_j|$, $\mu=(\mu_1,\ldots,\mu_M)\in \Lambda$ with $\mu_j=|\lambda_j|/A$, $z=(z_1,\ldots,z_M)\in\{-1,1\}^M$ with $z_j=\text{sgn}(\lambda_j)$. This new parametrization is identifiable and $(A,\mu,z)$ uniquely determines $\lambda$. Therefore, there exists a one-to-one correspondence between the prior on $(A,\mu,z)$ and the prior on $\lambda$.
Under this parametrization, the geometric properties of $\lambda$ transfer to those of $\mu$. For example, a prior on $\mu$ that induces sparsity will produce a sparse prior for $\lambda$. With this in mind, we propose the following double Dirichlet Gamma (DDG) prior for $\lambda$ or $(A, \mu, z)$:
\begin{align*}
     & \text{(DDG1)} & A\sim \text{Ga}(a_0, b_0), \ \mu\sim \text{Diri}\bigg(\frac{\alpha}{M^{\gamma}},\ldots,\frac{\alpha}{M^{\gamma}}\bigg),\
    \ z_1,\ldots,z_M\text{ iid with }P(z_i=1)=\frac{1}{2}. &&
\end{align*}
Since $\mu$ follows a Dirichlet distribution, it can be equivalently represented as
\begin{align*}
    \bigg(\frac{T_1}{\sum_{j=1}^pT_j},\ldots,\frac{T_M}{\sum_{j=1}^pT_j}\bigg),\ \text{with }T_j\overset{\text{iid}}{\sim}\text{Ga}\bigg(\frac{\alpha}{M^{\gamma}},1\bigg).
\end{align*}
Let $\eta=(\eta_1,\ldots,\eta_M)$ with $\eta_j=z_j\lambda_j$.
By marginalizing out the $z$, the prior for $\mu$ can be equivalently represented as
\begin{align}
    \bigg(\frac{T_1}{\sum_{j=1}^M|T_j|},\ldots,\frac{T_M}{\sum_{j=1}^M|T_j|}\bigg),\ \text{with }T_j\overset{\text{iid}}{\sim}\text{DG}\bigg(\frac{\alpha}{M^{\gamma}},1\bigg).\label{eq:dD}
\end{align}
where DG$(a,b)$ denotes the double Gamma distribution with shape parameter $a$, rate parameter $b$ and pdf $\{2\Gamma(a)\}^{-1}b^a|t|^{a-1}e^{-b|t|}$ ($t\in \bbR$), where $\Gamma(\cdot)$ is the Gamma function. More generally, we call a distribution as the double Dirichlet distribution with parameter $(a_1,\ldots,a_M)$, denoted by DD$(a_1,\ldots,a_M)$, if it can be represented by \eqref{eq:dD} with $T_j\sim $DG$(a_j,1)$. Then, the DDG prior for $\lambda$ has an alternative form as
\begin{align*}
  &\text{(DDG2)} &  \lambda=A\eta,\ A\sim\text{Ga}(a_0,b_0),\ \eta\sim \text{DD}\bigg(\frac{\alpha}{M^{\gamma}},\ldots,\frac{\alpha}{M^{\gamma}}\bigg).&&
\end{align*}
We will use the form (DDG2) for studying the theoretical properties of the DDG prior and focus on the form (DDG1) for posterior computation.

\section{Theoretical properties}
In this section, we study the prediction efficiency of the proposed Bayesian aggregation procedures for CA and LA in terms of convergence rate of posterior prediction.

We say that a Bayesian model $\mathcal{F}=\{P_{\theta}: \theta\in\Theta\}$, with a prior distribution $\Pi$ over the parameter space $\Theta$, has a posterior convergence rate at least $\epsilon_n$ if
\begin{align}\label{eq:dpcr}
    \Pi\big(d(\theta,\theta^\ast )\geq D\epsilon_n\big|X_1,\ldots,X_n\big)\overset{P_{\theta_0}}{\longrightarrow}0,
\end{align}
with a limit $\theta^\ast \in\Theta$, where $d$ is a metric on $\Theta$ and $D$ is a sufficiently large positive constant. For example, to characterize prediction accuracy, we use $d(\lambda,\lambda')=||f_{\lambda}-f_{\lambda'}||_Q$ and $||n^{-1/2}F(\lambda-\lambda')||_2$ for CA and LA, respectively. Let $P_0=P_{\theta_0}$ be the truth under which the iid observations $X_1,\ldots,X_n$ are generated. If $\theta_0\in\Theta$, then the model is well-specified and under mild conditions, $\theta^\ast =\theta_0$. If $\theta_0\notin\Theta$, then the limit $\theta^\ast $ is usually the point in $\Theta$ so that $P_{\theta}$ has the minimal Kullback-Leibler (KL) divergence to $P_{\theta_0}$.
\eqref{eq:dpcr} suggests that the posterior probability measure puts almost all its mass over a sequence of $d$-balls whose radii shrink towards $\theta^\ast $ at a rate $\epsilon_n$. In the following, we make the assumption that $\sigma$ is known, which is a standard assumption adopted in Bayesian asymptotic proofs to avoid long and tedious arguments. \cite{Jonge2013} studies the asymptotic behavior of the error standard deviation in regression when a prior is specified for $\sigma$. Their proofs can also be used to justify our setup when $\sigma$ is unknown. In the rest of the paper, we will frequently use $C$ to denote a constant, whose meaning might change from line to line.

\subsection{Posterior convergence rate of Bayesian convex aggregation}\label{se:pcrba}
Let $\Sigma=(E_Q[f_i(X)f_j(X)])_{M\times M}$ be the second order moment matrix of $(f_1(X),\ldots,f_M(X))$, where $X\sim Q$. Let $f^\ast =\sum_{j=1}^M\lambda_j^\ast f_{j}$  be the best $L_2(Q)$-approximation of $f_0$ in the space
$\mathcal{F}^{\Lambda}=\big\{f=\sum_{j=1}^M\lambda_jf_{j}: \lambda_j\geq0, \sum_{j=1}^M\lambda_j= 1\big\}$
of all convex combinations of $f_1,\ldots,f_M$, i.e. $\lambda^\ast=\text{arg}\min_{\lambda\in\Lambda}||f_{\lambda}-f_0||_Q^2$. This misspecified framework also includes the well-specified situation as a special case where $f_0=f^\ast\in\mathcal{F}^{\Lambda}$. Denote the $j$th column of $\Sigma$ by $\Sigma_j$.

We make the following assumptions:
\begin{description}
  \item[(A1)] There exists a constant $0<\kappa<\infty$ such that $\sup_{1\leq j\leq M}|\Sigma_{jj}|\leq\kappa$.
  \item[(A2)] (Sparsity) There exists an integer $s>0$, such that $||\lambda^\ast ||_0=s<n$.
  \item[(A3)] There exists a constant $0<\kappa<\infty$ such that $\sup_{1\leq j\leq M}\sup_{x\in\mathcal{X}}|f_j(x)|\leq\kappa$.
\end{description}

\begin{itemize}
  \item If $E_Q[f_j(X)]=0$ for each $j$, then $\Sigma$ is the variance covariance matrix. (A1) assumes the second moment $\Sigma_{jj}$ of $f_j(X)$ to be uniformly bounded. By applying Cauchy's inequality, the off-diagonal elements of $\Sigma$ can also be uniformly bounded by the same $\kappa$.
  \item (A3) implies (A1). This uniformly bounded condition is only used in Lemma \ref{le:test} part a. As illustrated by \cite{Birge2004}, such a condition is necessary for studying the $L_2(Q)$ loss of Gaussian regression with random design, since under this condition the Hellinger distance between two Gaussian regression models is equivalent to the $L_2(Q)$ distance between their mean functions.
  \item Since $\lambda^\ast \in \Lambda$, the $l_1$ norm of $\lambda^\ast $ is always equal to one, which means that $\lambda^\ast $ is always $l_1$-summable. $(A2)$ imposes an additional sparse structure on $\lambda^\ast $. We will study separately the convergence rates with and without (A2). It turns out that the additional sparse structure improves the rate if and only if $s\ll \sqrt{\frac{n}{\log M}}$.
\end{itemize}

The following theorem suggests that the posterior of $f_{\lambda}$ concentrates on an $||\cdot||_Q$-ball around the best approximation $f^\ast$ with a radius proportional to the minimax rate of CA. In the special case when $f_\ast=f_0$, the theorem suggests that the proposed Bayesian procedure is minimax optimal.

\begin{theorem}\label{thm:BA}
Assume (A3). Let $(X_1,Y_1),\ldots,(X_n,Y_n)$ be $n$ iid copies of $(X,Y)$ sampled from $X\sim Q$, $Y|X\sim N(f_0(X),\sigma^2)$. If $f^\ast =\sum_{j=1}^M\lambda_j^\ast f_j$ is the minimizer of $f\mapsto||f-f_0||_Q$ on $\mathcal{F}^{\Lambda}$, then under the prior (DA), for some $D>0$, as $n\to\infty$,
\begin{align*}
    E_{0,Q} \Pi\bigg(||f-f^\ast ||_Q\geq D\min\bigg\{\sqrt{\frac{M}{n}},\sqrt[4]{\frac{\log(M/\sqrt{n}+1)}{n}}\bigg\}\bigg|X_1,Y_1,\ldots,X_n,Y_n\bigg)\to 0.
\end{align*}
Moreover, if (A2) is also satisfied, then as $n\to\infty$,
\begin{align*}
    E_{0,Q} \Pi\bigg(||f-f^\ast ||_Q\geq D\sqrt{\frac{s\log(M/s)}{n}}\ \bigg|X_1,Y_1,\ldots,X_n,Y_n\bigg)\to 0.
\end{align*}
\end{theorem}

\subsection{Posterior convergence rate of Bayesian linear aggregation}
Let $\lambda^\ast =(\lambda_1^\ast ,\ldots,\lambda_M^\ast )$ be the coefficient such that $F\lambda^\ast$ best approximates $F_0$ in $||\cdot||_2$ norm, i.e. $\lambda^\ast=\text{arg}\min_{\lambda\in\bbR^M}||F\lambda-F_0||_2^2$. Similar to the CA case, such a misspecified framework also includes the well-specified situation as a special case where $F_0=F\lambda^\ast\in\mathcal{F}^{\bbR^M}$. It is possible that there exists more than one such a minimizer and then we can choose $\lambda^\ast$ with minimal nonzero components. This non-uniqueness will not affect our theorem quantifying the prediction performance of LA since any minimizers of $||F\lambda-F_0||_2^2$ will give the same prediction $F\lambda$. Our choice of $\lambda^\ast$, which minimizes $||\lambda^\ast||_0$, can lead to the fastest posterior convergence rate.

We make the following assumptions:
\begin{description}
  \item[(B1)] There exists a constant $0<\kappa<\infty$ such that $\frac{1}{\sqrt{n}}\sup_{1\leq j\leq M}||F_j||_2\leq\kappa$.
  \item[(B2a)] (Sparsity) There exists an integer $s>0$, such that $||\lambda^\ast ||_0=s<n$.
  \item[(B2b)] ($l_1$-summability) There exists a constant $A_0>0$, such that $A^\ast =||\lambda^\ast ||_1<A_0$.
  \item[(B3)] For $m_0=\lceil\sqrt{n}\, \rceil$, there exists a constant $\kappa_0>0$ such that $\frac{1}{\sqrt{n}}||F\lambda||_2\geq \kappa_0||\lambda||_1$ for all $\lambda\in \bbR^M$ with  $||\lambda||_0= m_0$.
\end{description}

\begin{itemize}
  \item (B1) is the column normalizing condition for the design matrix. This assumption is mild since the predictors can always be normalized to satisfy it. This condition can also be considered as the empirical version of (A1), where the matrix $\Sigma$ is replaced by its empirical estimator $\frac{1}{n}F^TF$.
  \item (B2a) is a counterpart of the sparsity condition (A2) of the aggregation problem. This assumption is commonly made in the high dimensional linear regression literature. (B2b) is assumed by \citet{Peter2006} in studying consistency of boosting for high dimensional linear regression. This condition includes the sparsity condition (B2a) as a special case while also including the case in which many components of $\lambda^\ast $ are nonzero but small in magnitude. Similar to the aggregation problem, under (B2b), the sparsity gains only when $s\ll \sqrt{\frac{n}{\log M}}$. (B2a) also implies a sparsity constraint on $\eta^\ast =\lambda^\ast /A^\ast $, where $\eta^\ast $ always satisfies $||\eta^\ast ||_1=1$.
  \item (B3) is the same in spirit as the sparse eigenvalue condition made in \cite{Raskutti2011}, which provides identifiability for $m_0$-sparse vectors. This assumption is only made for the $l_1$-summable case, where any $l_1$-summable $\lambda\in\bbR^M$ can be approximated by an $m_0$-sparse vector with error at most $O(||\lambda_1||\epsilon_n)$ under $d_F$ (Lemma \ref{le:coen} part b), with $\epsilon_n$ given in (DA-PC), where $d_F(\lambda,\lambda')=||n^{-1/2}F(\lambda-\lambda')||_2$. Under this assumption, we show that the posterior probability of $\{||\lambda||_1\leq KA^\ast \}$ converges to zero as $n\to\infty$ for some constant $K$ and therefore with high posterior probability, $\lambda$ can be approximated by an $m_0$-sparse vector with error at most $O(\epsilon_n)$.
\end{itemize}

The following theorem is a counterpart of Theorem \ref{thm:BA} for LA.

\begin{theorem}\label{thm:LR}
Assume (B1). Let $Y$ be an $n$-dimensional response vector sampled from $Y\sim N(F_0,\sigma^2I_n)$. Let $\lambda^\ast$ be any one of the minimizers of $\lambda\mapsto||F\lambda-F_0||_2$ in $\bbR^M$. If (B2b) and (B3) are true, then under the prior (DDG2), for some $D>0$, as $n\to\infty$,
\begin{align*}
    E_0 \Pi\bigg(||n^{-\frac{1}{2}}F(\lambda-\lambda^\ast)||_2\geq D\min\bigg\{\sqrt{\frac{M}{n}},\sqrt[4]{\frac{\log(M/\sqrt{n}+1)}{n}}\bigg\}\bigg|\ Y\bigg)\to 0.
\end{align*}
If (B2a) is true, then as $n\to\infty$,
\begin{align*}
    E_0 \Pi\bigg(||n^{-\frac{1}{2}}F(\lambda-\lambda^\ast)||_2\geq D\sqrt{\frac{s\log(M/s)}{n}}\ \bigg|\ Y\bigg)\to 0.
\end{align*}
\end{theorem}

Theorem \ref{thm:LR} suggests that in order to obtain the fastest posterior convergence rate for prediction, we can choose the $\lambda^\ast$ having the minimal $||\lambda^\ast||_0$ among all minimizers of $||F\lambda-F_0||_2$. This suggests that the posterior measure tends to concentrate on the sparsest $\lambda^\ast$ that achieves the same prediction accuracy, which explains the sparse adaptivity. The non-uniqueness happens when $M>n$.

\section{Experiments}
As suggested by \cite{Yuhong2001}, the estimator $\hat{f}_n$ depends on the order of the observations and one can randomly permute the order a number of times and average the corresponding estimators. In addition, one can add a third step of estimating $f_1,\ldots,f_M$ with the full dataset as $\hat{f}_1,\ldots,\hat{f}_M$ and setting the final estimator as $\tilde{f}=\sum_{j=1}^M\hat{\lambda}_j\hat{f}_j$.
We will adopt this strategy and our splitting and aggregation scheme can be summarized as follows. First, we randomly divide the entire $n$ samples into two subsets $S_1$ and $S_2$ with $|S_1|=n_1$ and $|S_2|=n_2$. As a default, we set $n_1=0.75n$ and $n_2=0.25n$. Using $S_1$ as a training set, we obtain $M$ base learners $\hat{f}_1^{(n_1)},\ldots, \hat{f}_M^{(n_1)}$. Second, we apply the above MCMC algorithms to aggregate these learners into $\hat{f}^{(n_1)}=\sum_{j=1}^M\hat{\lambda}_j\hat{f}_j^{(n_1)}$ based on the $n_2$ aggregating samples. Finally, we use the whole dataset to train these base learners, which gives us $\hat{f}_j^{(n)}$, and the final estimator is $\hat{f}^{(n)}=\sum_{j=1}^M\hat{\lambda}_j\hat{f}_j^{(n)}$. Therefore, one basic requirement on the base learners is that they should be stable in the sense that $\hat{f}_j^{(n)}$ can not be dramatically different from $\hat{f}_j^{(n_1)}$ (e.g. CART might not be a suitable choice for the base learner).

\subsection{Bayesian linear aggregation}
In this subsection, we apply the Bayesian LA methods to the linear regression $Y=X\lambda+\epsilon$, with $X\in\bbR^M$ and $\epsilon\sim N(0,\sigma^2I_n)$. Since every linear aggregation problem can be reformed as a linear regression problem, this is a simple canonical setting for testing our approach. We consider two scenarios: 1. the sparse case where the number of nonzero components in the regression coefficient $\lambda$ is smaller than $M$ and the sample size $n$; 2. the non-sparse case where $\lambda$ can have many nonzero components, but the $l_1$ norm $||\lambda||_1=\sum_{j=1}^M|\lambda_j|$ remains constant as $M$ changes. We vary model dimensionality by letting $M=5$, $20$, $100$ and $500$.

We compare the Bayesian LA methods with lasso, ridge regression and horseshoe. Lasso \citep{Tibshirani1996} is widely used for linear models, especially when $\lambda$ is believed to be sparse. In addition, due to the use of $l_1$ penalty, the lasso is also minimax optimal when $\lambda$ is $l_1$-summable \citep{Raskutti2011}. Ridge regression \citep{Hoerl1970} is a well-known shrinkage estimator for non-sparse settings. Horseshoe \citep{Carvalho2010} is a Bayesian continuous shrinkage prior for sparse regression from the family of global-local mixtures of Gaussians \citep{Polson2010}. Horseshoe is well-known for its robustness and excellent empirical performance for sparse regression, but there is a lack of theoretical justification. $n$ training samples are used to fit the models and $N-n$ testing samples are used to calculate the prediction root mean squared error (RMSE) $\big\{(N-n)^{-1}\sum_{i=n+1}^N(\hat{y}_i-y_i)^2\big\}^{1/2}$, where $\hat{y}_i$ denotes the prediction of $y_i$.

The MCMC algorithm for the Bayesian LA method is run for 2,000 iterations, with the first 1,000 iterations as the burn-in period. We set $\alpha=1$, $\gamma=2$, $a_0=0.01$ and $b_0=0.01$ for the hyperparameters. The tuning parameters in the MH steps are chosen so that the acceptance rates are around $40\%$.
The lasso is implemented by the \texttt{glmnet} package in \texttt{R}, the ridge is implemented by the \texttt{lm.ridge} function in \texttt{R} and horseshoe is implemented by the \texttt{monomvn} package in \texttt{R}. The iterations for horseshoe is set as the default 1,000. The regularization parameters in Lasso and ridge are selected via cross-validation.

\subsubsection{Sparse case}
In the sparse case, we choose the number of non-zero coefficients to be $5$. The simulation data are generated from the following model:
\begin{align*}
  &(S)&  y=-0.5x_1+x_2+0.4x_3-x_4+0.6x_5+\epsilon,\quad \epsilon\sim N(0,0.5^2),&
\end{align*}
with $M$ covariates $x_1,\ldots,x_M\sim$ i.i.d $N(0,1)$. The training size is set to be $n=100$ and testing size $N-n=1000$. As a result, (S) with $M=5$ and $20$ can be considered as moderate dimensional, while $M=100$ and $M=500$ are relatively high dimensional.

\begin{table*}[h!]
  \centering
\begin{tabular}{c|C{2cm}C{2cm}C{2cm}C{2cm}}
\hhline{=====}
   $M$ & 5 & 20 & 100 & 500 \\
\hline
       \multirow{2}{*}{LA} & .511 & .513 & .529 & .576  \\
       & (0.016) & (0.016) & (0.020) & (0.023)  \\
       \hline
      \multirow{2}{*}{Lasso}& .514 & .536 & .574 & .613 \\
      & (0.017) & (0.020) & (0.039) & (0.042)  \\
       \hline
       \multirow{2}{*}{Ridge}& .514 & .565 & 1.23 & 2.23  \\
       &(0.017) & (0.019) & (0.139) & (0.146)  \\
       \hline
       \multirow{2}{*}{Horseshoe}& .512 & .519 & .525 & .590  \\
       &(0.016) & (0.014) & (0.019) & (0.022)  \\
       \hhline{=====}
\end{tabular}
  \caption{RMSE for the sparse linear model (S). The numbers in the parentheses indicate the standard deviations. All results are based on 100 replicates.}\label{table:slr}
\end{table*}

From Table \ref{table:slr}, all the methods are comparable when there is no nuisance predictor ($M=5$). However, as more nuisance predictors are included, the Bayesian LA method and horseshoe have noticeably better performance than the other two methods. For example, for $M=100$, the Bayesian LA method has $8\%$ and $53\%$ improvements over lasso and ridge, respectively. In addition, as expected, ridge deteriorates more dramatically than the other two as $M$ grows. It appears that Bayesian LA is more computationally efficient than horseshoe. For example, under $m=100$ it takes horseshoe 50 seconds to draw 1,000 iterations but only takes LA about 1 second to draw 2,000 iterations.

\begin{figure}[htp]
\centering
  \includegraphics[width=6in]{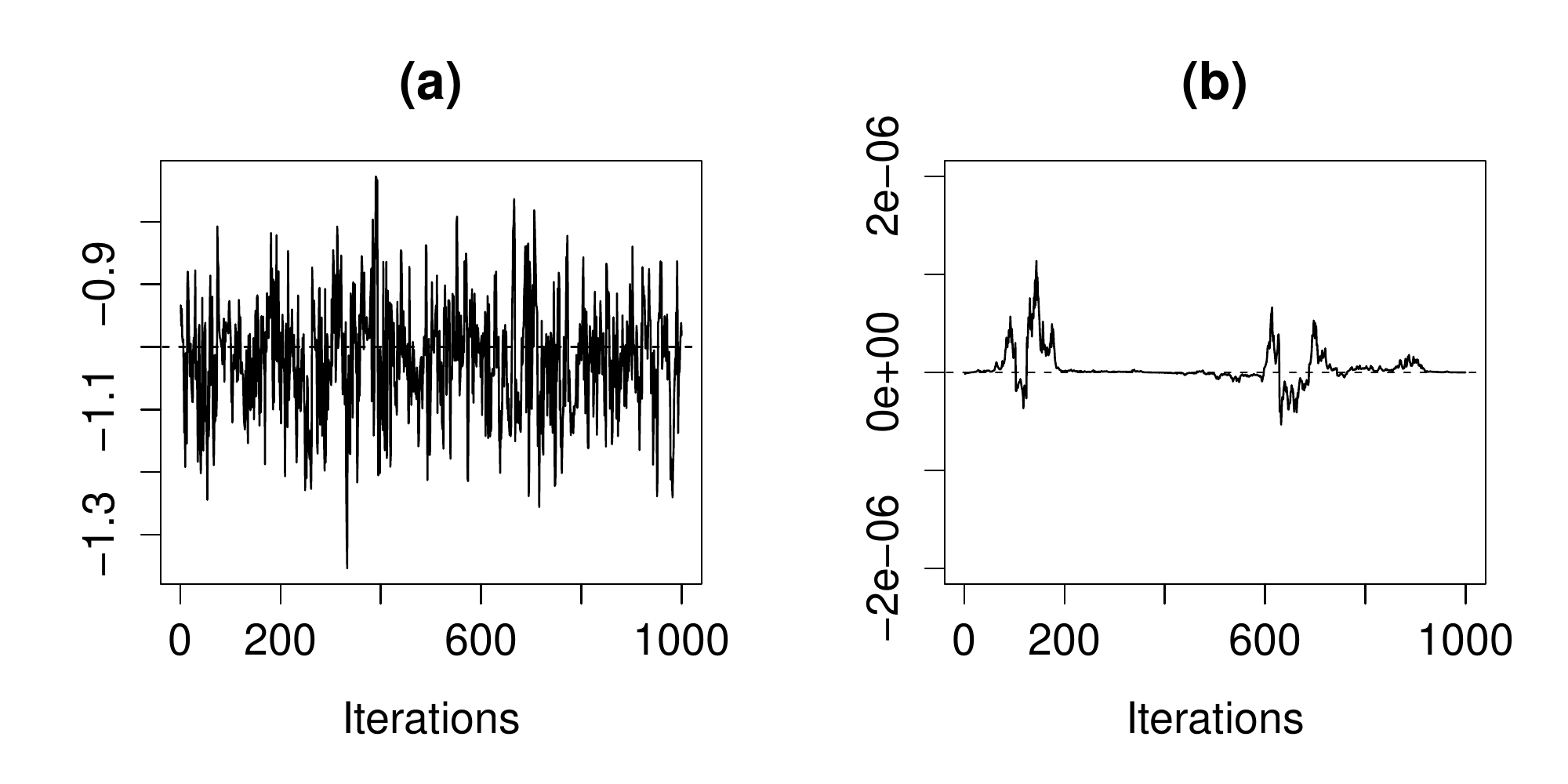}\\
  \caption{Traceplots for a non-zero regression coefficient and a zero coefficient.}\label{fig:2}
\end{figure}

\begin{figure}[htp]
\centering
  \includegraphics[width=6in]{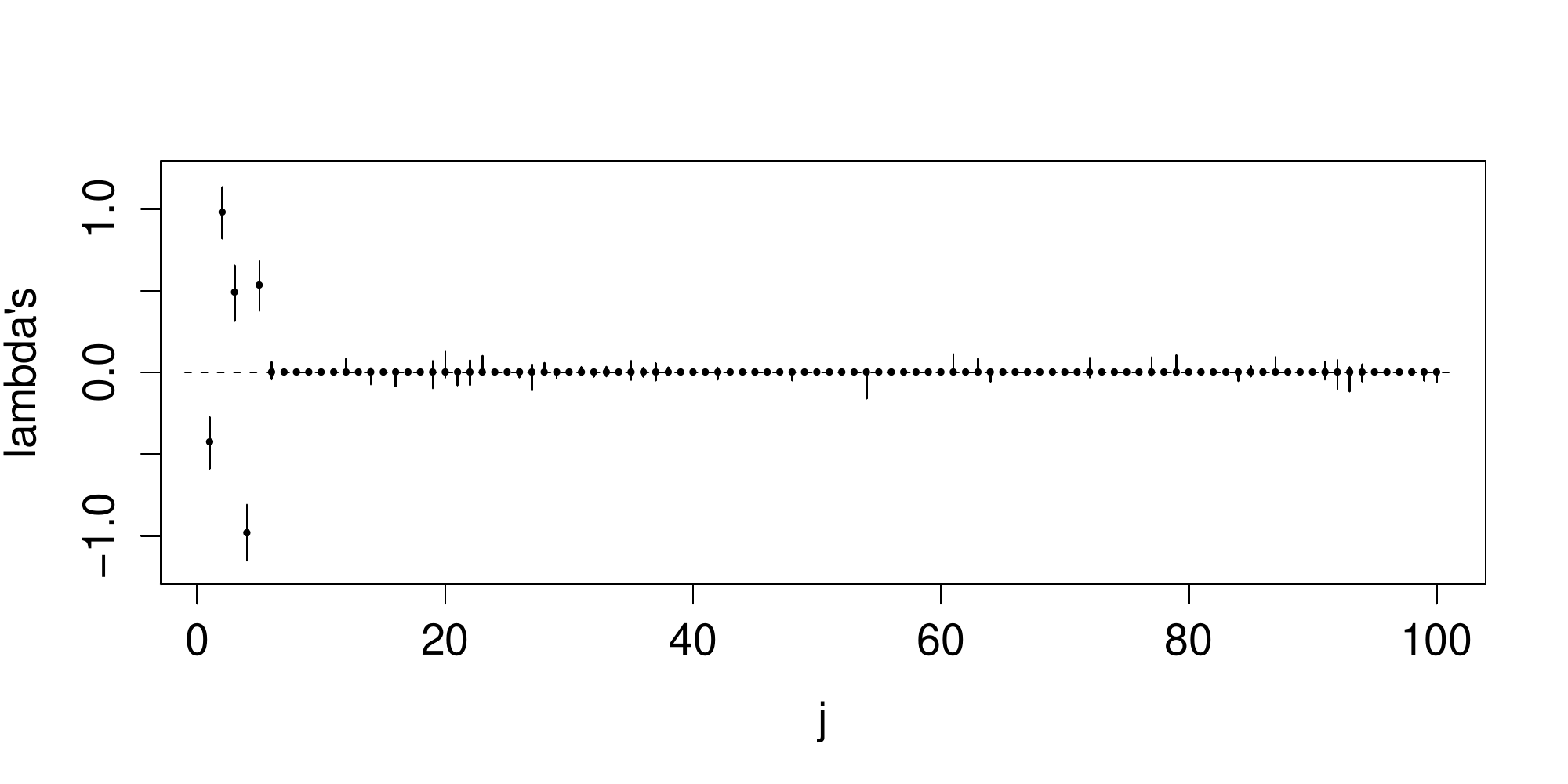}\\
  \caption{95\% posterior credible intervals for $\lambda_1,\ldots,\lambda_{100}$ in sparse regresion. The solid dots are the corresponding posterior medians.}\label{fig:3}
  \includegraphics[width=6in]{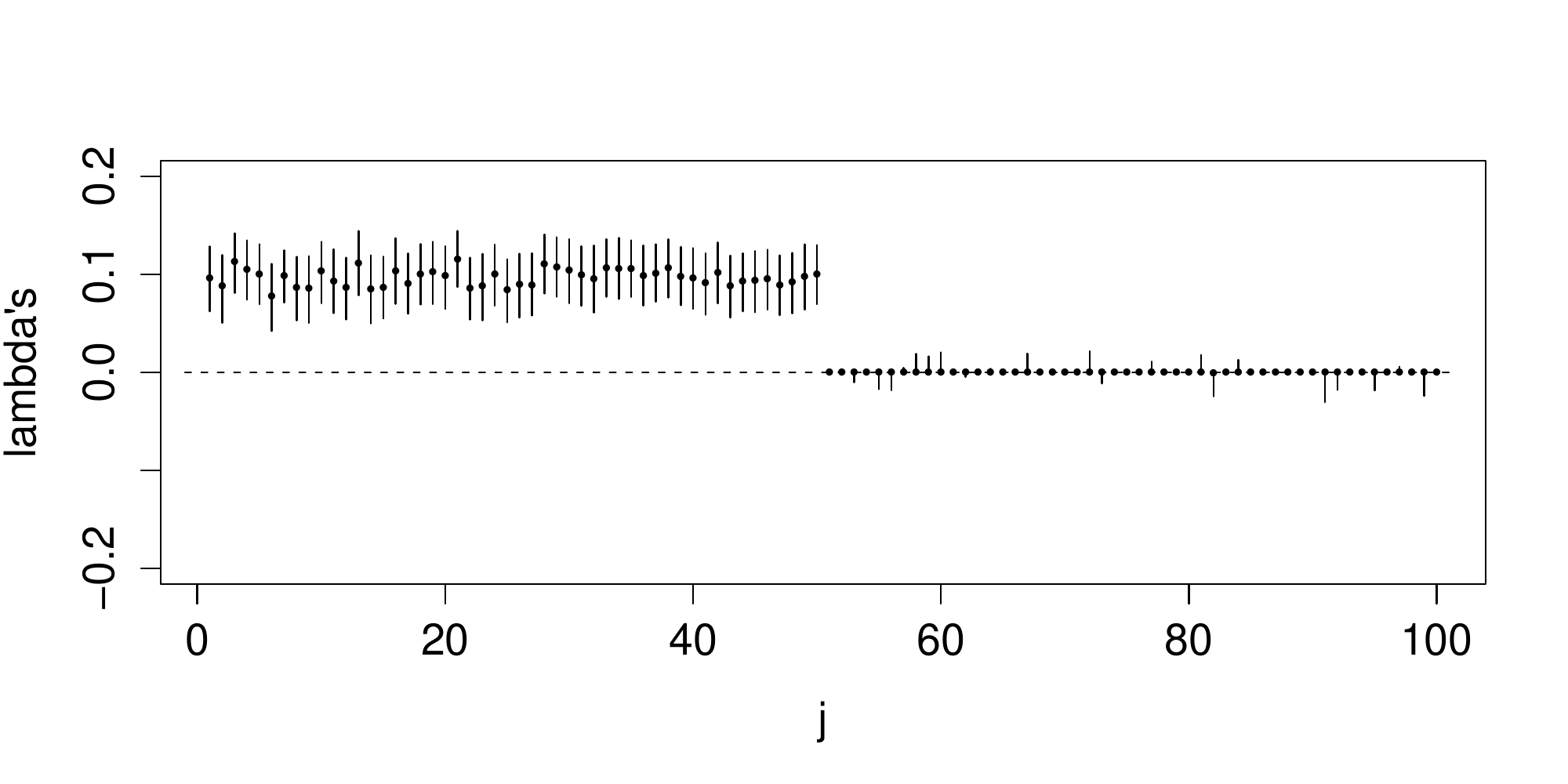}\\
  \caption{95\% posterior credible intervals for $\lambda_1,\ldots,\lambda_{100}$ in non-sparse regression. The solid dots are the corresponding posterior medians.}\label{fig:4}
\end{figure}

Fig.~\ref{fig:2} displays the traceplots after the burn-in for a typical non-zero and a typical zero regression coefficient respectively under $M=100$. The non-zero coefficient mixes pretty well according to its traceplot. Although the traceplot of the zero coefficient exhibits some small fluctuations, their magnitudes are still negligible compared to the non-zero ones. We observe that these fluctuant traceplots like Fig.~\ref{fig:2}(b) only happens for those $\lambda_j$'s whose posterior magnitudes are extremely small. The typical orders of the posterior means of those $\lambda_j$'s in LA that correspond to unimportant predictors range from $10^{-17}$ to $10^{-2}$. However, the posterior medians of unimportant predictors are less than $10^{-4}$ (see Fig.~\ref{fig:3}). This suggests that although the coefficients are not exactly to zero, the estimated regression coefficients with zero true values are still negligible compared to the estimators of the nonzero coefficients. In addition, for LA the posterior median appears to be a better and more robust estimator for sparse regression than the posterior mean.

\subsubsection{Non-sparse case}
In the non-sparse case, we use the following two models as the truth:
\begin{align*}
    &(NS1) & y=\sum_{j=1}^M\frac{3(-1)^j}{j^2}x_j+\epsilon,\quad \epsilon\sim N(0,0.1^2),&&\\
    &(NS2) & y=\sum_{j=1}^{\lfloor M/2\rfloor}\frac{5}{\lfloor M/2\rfloor}x_j+\epsilon,\quad \epsilon\sim N(0,0.1^2),&&
\end{align*}
with $M$ covariates $x_1,\ldots,x_M\sim$ i.i.d $N(0,1)$. In (NS1), all the predictors affect the response and the impact of predictor $x_j$ decreases quadratically in $j$. Moreover, $\lambda$ satisfies the $l_1$-summability since $\lim_{p\to\infty}||\lambda||_1=\pi^2/3\thickapprox 4.9$. In (NS2), half of the predictors have the same influence on the response with $||\lambda||_1=5$.  The training size is set to be $n=200$ and testing size $N-n=1000$ in the following simulations.

\begin{table*}[h!]
  \centering
\begin{tabular}{c|c|C{2cm}C{2cm}C{2cm}C{2cm}}
\hhline{======}
   & $M$ & 5 & 20 & 100 & 500 \\
\hline
      \multirow{6}{*}{NS1}&  \multirow{2}{*}{LA} & .101 & .112 & .116 & .129  \\
      & & (0.002) & (0.003) & (0.005) & (0.007)  \\
       \hhline{~-----}
      & \multirow{2}{*}{Lasso}& .105 & .110 & .116 & .155 \\
      && (0.006) & (0.005) & (0.005) & (0.006)  \\
       \hhline{~-----}
      & \multirow{2}{*}{Ridge}& .102 & .107 & .146 & 2.42  \\
      & &(0.003) & (0.004) & (0.008) & (0.053)  \\
      \hhline{~-----}
      & \multirow{2}{*}{Horseshoe}& .102 & .111 & .114 & .136  \\
      & &(0.003) & (0.003) & (0.004) & (0.005)  \\
       \hline
       \multirow{6}{*}{NS2}&  \multirow{2}{*}{LA} & .101 & .104 & .121 & .326  \\
      & & (0.002) & (0.003) & (0.005) & (0.008)  \\
       \hhline{~-----}
      & \multirow{2}{*}{Lasso}& .111 & .106 & .131 & .323 \\
      && (0.006) & (0.003) & (0.007) & (0.008)  \\
       \hhline{~-----}
      & \multirow{2}{*}{Ridge}& .103 & .107 & .140 & .274  \\
      & &(0.003) & (0.003) & (0.008) & (0.010)  \\
      \hhline{~-----}
      & \multirow{2}{*}{Horseshoe}& .102 & .104 & .124 & .308  \\
      & &(0.003) & (0.003) & (0.004) & (0.007)  \\
       \hhline{======}
\end{tabular}
  \caption{RMSE for the non-sparse linear models (NS1) and (NS2). All results are based on 100 replicates.}\label{table:nslr}
\end{table*}

From Table \ref{table:nslr}, all the methods have comparable performance when $M$ is moderate (i.e $5$ or $20$) in both non-sparse settings. In the non-sparse settings, horseshoe also exhibits excellent prediction performance. In most cases, LA, lasso and horseshoe have similar performance. As $M$ increases to an order comparable to the sample size, LA and horseshoe tend to be more robust than lasso and ridge. As $M$ becomes much greater than $n$, LA, lasso and horseshoe remain good in (NS1) while breaking down in (NS2); ridge breaks down in (NS1) while becoming the best in (NS2). It might be because in (NS1), although all $\lambda_j$'s are nonzero, the first several predictors still dominate the impact on $y$. In contrast, in (NS2), half of $\lambda_j$'s are nonzero and equally small. Fig.~\ref{fig:4} plots 95\% posterior credible intervals for $\lambda_1,\ldots,\lambda_{100}$ of (NS2) under $M=100$. According to Section \ref{se:minimax}, the spasrse/non-sparse boundary for $(NS2)$ under $M=100$ is $\sqrt{200}/\log{100}\approx 3\ll 50$. Therefore, the results displayed in Fig.~\ref{fig:4} can be classified into the non-sparse regime.
A simple variable selection based on these credible intervals correctly identifies all $50$ nonzero components.

\subsubsection{Robustness against the hyperparameters}
Since changing the hyperparameter $\alpha$ in the Dirichlet prior is equivalent to changing the hyperparameter $\gamma$, we perform a sensitivity analysis for $\gamma$ in the above two regression settings with $M=100$.

\begin{figure}[htp]
\centering
\begin{tabular}{cc}
    \includegraphics[scale=0.36]{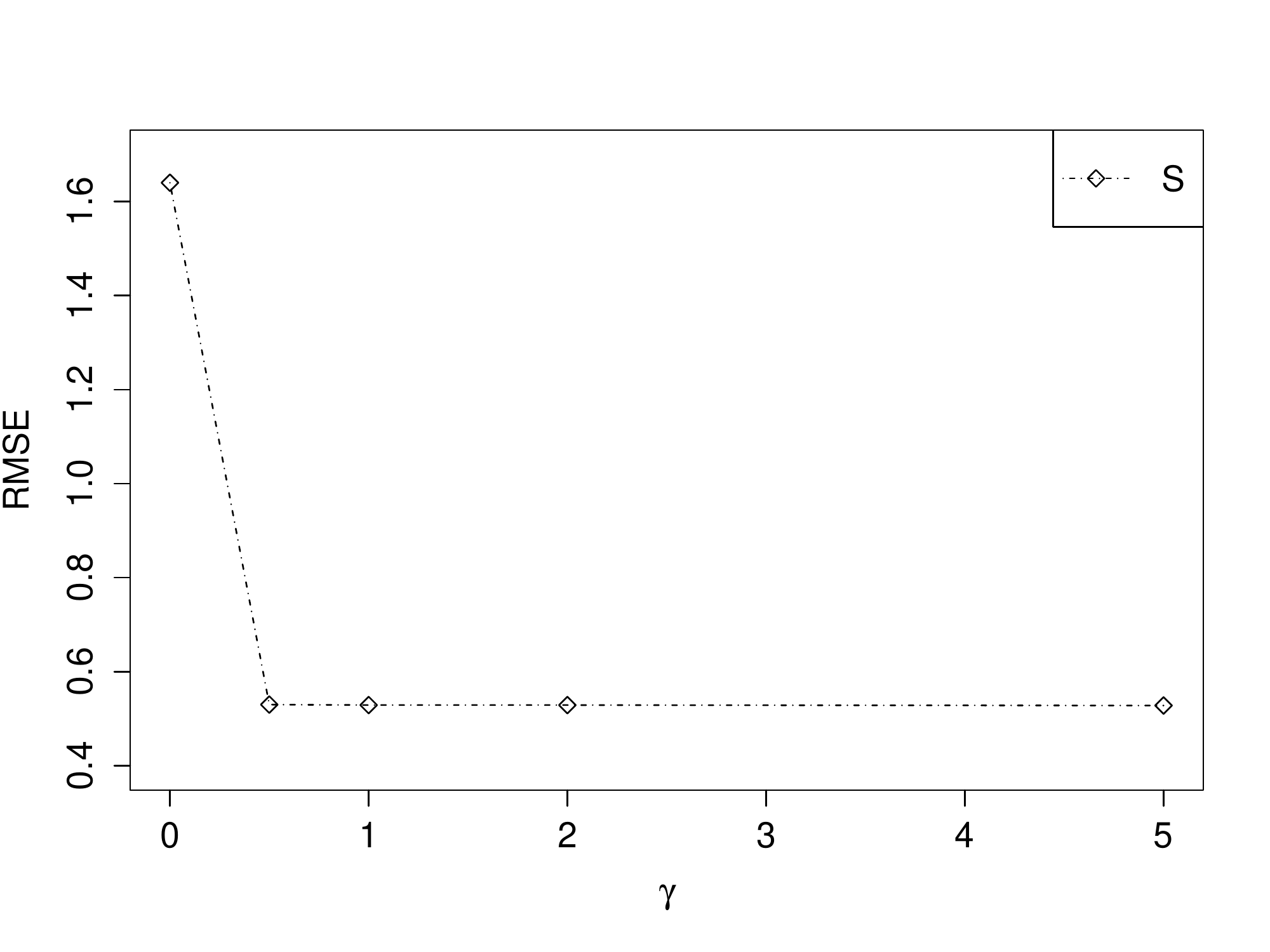}
    &
    \includegraphics[scale=0.36]{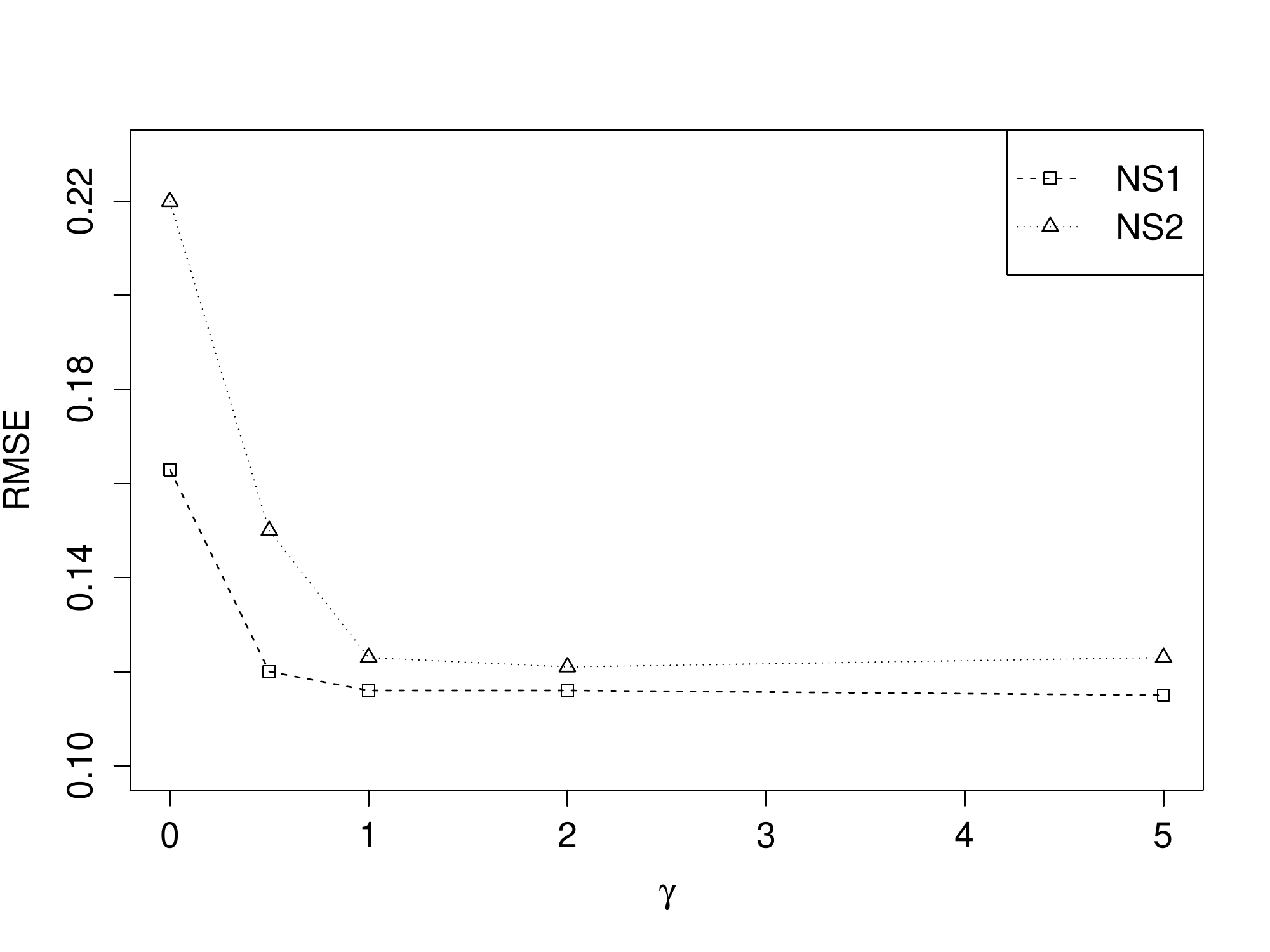}
\end{tabular}
\caption{Robustness of the Bayesian LA methods against the hyperparameter $\gamma$. The results are based on 100 replicates.}
\label{fig:rob}
\end{figure}

From Figure \ref{fig:rob}, the Bayesian LA method tends to be robust against the change in $\gamma$ at a wide range. As expected, the Bayesian LA method starts to deteriorate as $\gamma$ becomes too small. In particular, when $\gamma$ is zero, the Dirichlet prior no longer favors sparse weights and the RMSE becomes large (especially for the sparse model) in all three settings. However, the Bayesian LA methods tend to be robust against increase in $\gamma$. As a result, we would recommend choosing $\gamma=2$ in practice.

\subsection{Bayesian convex aggregation}
In this subsection, we conduct experiments for the Bayesian convex aggregation method.

\subsubsection{Simulations}
The following regression model is used as the truth in our simulations:
\begin{align}\label{eq:sim1}
    y=x_1+x_2+3x_3^2-2e^{-x_4}+\epsilon,\quad \epsilon\sim N(0,0.5),
\end{align}
with $p$ covariates $x_1,\ldots,x_d\sim$ i.i.d $N(0,1)$. The training size is set to be $n=500$ and testing size $N-n=1000$ in the following simulations.

In the first simulation, we choose $M=6$ base learners: CART, random forest (RF), lasso, SVM, ridge regression (Ridge) and neural network (NN). The Bayesian aggregation (BA) is compared with the super learner (SL). SL is implemented by the \texttt{SuperLearner} package in \texttt{R}. The implementations of the base learners are described in Table \ref{table:packages}. The MCMC algorithm for the Bayesian CA method is run for 2,000 iterations, with the first 1,000 iterations as the burn-in period. We set $\alpha=1$, $\gamma=2$ for the hyperparameters. The simulation results are summarized in Table \ref{table:1}, where square roots of mean squared errors (RMSE) of prediction based on 100 replicates are reported.

\begin{table*}[h!]
  \centering
\begin{tabular}{C{3cm}|C{3cm}|C{3cm}|C{3cm}}
  \hhline{====}
   Base learner & CART & RF & Lasso \\
    \hline
       \texttt{R} package & \texttt{rpart} & \texttt{randomForest} & \texttt{glmnet} \\
      \hhline{====}
     SVM & Ridge & NN & GAM \\
       \hline
       \texttt{e1071} & \texttt{MASS} & \texttt{nnet} & \texttt{gam} \\
       \hhline{====}
\end{tabular}
  \caption{Descriptions of the base learners. }\label{table:packages}
\end{table*}

\begin{table*}[h!]
  \centering
\begin{tabular}{c|C{1.2cm}C{1.2cm}C{1.2cm}C{1.2cm}C{1.2cm}C{1.2cm}C{1.2cm}C{1.2cm}}
\hhline{=========}
   $d$ & CART & RF & Lasso & SVM & Ridge & NN & SL & BA \\
\hline
       \multirow{2}{*}{5} & 3.31 & 3.33 & 5.12 & 2.71 & 5.12 & 3.89 & 2.66 & \textbf{2.60} \\
       & (0.41) & (0.42) & (0.33) & (0.49) & (0.33) & (0.90) & (0.48) & (0.48) \\
       \hline
      \multirow{2}{*}{20}& 3.32& 3.11 & 5.18 & 4.10 & 5.23 & 5.10 & 3.13 & \textbf{3.00} \\
      & (0.41) & (0.49) & (0.37) & (0.46) & (0.38) & (1.57) & (0.54) & (0.48) \\
       \hline
       \multirow{2}{*}{100}& 3.33 & 3.17 & 5.17 & 5.48 & 5.64 & 7.12 & 3.19 & \textbf{3.03} \\
       &(0.38) & (0.45) & (0.32) & (0.35) & (0.33) & (1.31) & (0.45) & (0.45) \\
       \hhline{=========}
\end{tabular}
  \caption{RMSE for the first simulation. All results are based on 100 replicates.}\label{table:1}
\end{table*}

In the second simulation, we consider the case when $M$ is moderately large. We consider $M=26$, $56$ and $106$ by introducing $(M-6)$ new base learners in the following way. In each simulation, for $j=1,\ldots, M-6$, we first randomly select a subset $S_j$ of the covariates $\{x_1,\ldots,x_d\}$ with size $p=\lfloor\min\{n^{1/2}, {d/3}\}\rfloor$. Then the $j$th base learner $f_j$ is fitted by the general additive model (GAM) with the response $y$ and covariates in $S_j$ as predictors. This choice of new learners is motivated by the fact that the truth is sparse when $M$ is large and brutally throwing all covariates into the GAM tends to have a poor performance. Therefore, we expect that a base learner based on GAM that uses a small subset of the covariates containing the important predictors $x_1$, $x_2$ $x_3$ and $x_4$ tends to have better performance than the full model. In addition, with a large $M$ and moderate $p$, the probability that one of the randomly selected $(M-6)$ models contains the truth is high. In this simulation, we compare BA with SL and a voting method using the average prediction across all base learners. For illustration, the best prediction performance among the $(M-6)$ random-subset base learners is also reported. Table \ref{tab:2} summarizes the results.

\begin{table*}[h!]
  \centering
\begin{tabular}{c|c|C{1.7cm}C{1.7cm}C{1.7cm}C{1.7cm}}
\hhline{======}
   $d$ & $M$ & Best & Voting & SL & BA \\
\hline
         \multirow{6}{*}{$20$} & \multirow{2}{*}{26}  & 3.14 & 4.63 & 3.40 & \textbf{2.78} \\
       & & (0.82) & (0.48) & (0.60) & (0.52) \\
       \hhline{~-----}
      & \multirow{2}{*}{56}& 2.86 & 4.98 & 3.33 & \textbf{2.79} \\
      & & (1.57) & (0.71) & (0.86) & (0.78) \\
        \hhline{~-----}
      & \multirow{2}{*}{106} & 2.79 & 4.95 & 3.23 & \textbf{2.73} \\
       & & (0.62) & (0.60) & (0.70) & (0.61) \\
       \hline
        \multirow{6}{*}{$100$} & \multirow{2}{*}{26}  & 3.14 & 4.72 & 3.09 & \textbf{2.78} \\
       & & (0.71) & (0.44) & (0.50) & (0.45) \\
       \hhline{~-----}
      & \multirow{2}{*}{56}& 2.95 & 4.93 & 3.07 & \textbf{2.78} \\
      & & (0.46) & (0.45) & (0.50) & (0.47) \\
        \hhline{~-----}
      & \multirow{2}{*}{106} & 2.84 & 4.90 & 2.98 & \textbf{2.69} \\
       & & (0.45) & (0.47) & (0.55) & (0.51) \\
       \hline
          \multirow{6}{*}{$500$} & \multirow{2}{*}{26}  & 5.21 & 5.75 & 3.77 & \textbf{3.19} \\
       & & (0.75) & (0.50) & (0.650) & (0.59) \\
       \hhline{~-----}
      & \multirow{2}{*}{56}& 4.86 & 5.92 & 4.02 & \textbf{3.18} \\
      & & (0.78) & (0.59) & (0.73) & (0.70) \\
        \hhline{~-----}
      & \multirow{2}{*}{106} & 4.65 & 5.98 & 4.18 & \textbf{3.13} \\
       & & (0.69) & (0.45) & (0.52) & (0.49) \\
       \hhline{======}
\end{tabular}
  \caption{RMSE for the second simulation study. All results are based on 100 replicates.}\label{tab:2}
\end{table*}

%Table \ref{tab:3} summarizes the running times of BA and SL in each combinations of $M$ and $p$. The weight of BA is computed in the similar spirit of 3-fold cross validation, where two-thirds of data are used to fit the base learners and one-third is used to learner the aggregation weight and the final weight is the average of the three alterations with different aggregation data. It turns out that when $p$ is large, a considerable part of the computation of BA is cost by fitting the base learners and the aggregation step is relatively fast.
%
%\begin{table*}[h!]
%  \centering
%\begin{tabular}{c|C{3cm}C{3cm}C{3cm}C{3cm}}
%  % after \\: \hline or \cline{col1-col2} \cline{col3-col4} ...
%  \hline
%   &  $p=5$ &  $p=20$ &$p=100$ &$p=500$ \\
%   \hhline{=|====}
%   $m=6$ &  $11/6(8)$   &   $10/9(8)$  &  $19/39(8)$   &  -   \\
%   $m=26$ &  -   &  $41/20(27)$   & $66/101(27)$   &   $120/273(28)$  \\
%   $m=56$ &  -  &  $83/34(58)$   &  $127/172(60)$   &   $154/274(59)$  \\
%   $m=106$ &  -  &  $175/79(112)$   &  $239/339(110)$   &  $258/408(110)$   \\
%   \hhline{=|====}
%\end{tabular}
%  \caption{Time consumptions (seconds) of $BA$ (left numbers) and $SL$ (right numbers) in the simulation study. The numbers in the parentheses indicate the aggregation time for BA. All results are based on 100 replicates.}\label{tab:3}
%\end{table*}

\subsubsection{Applications}
We apply BA to four datasets from the UCI repository. Table \ref{table:data} provides a brief description of those datasets.  We use CART, random forest, lasso, support vector machine, ridge regression and neural networks as the base learners. We run 40,000 iterations for the MCMC for the BA for each dataset and discard the first half as the burn-in. Table \ref{table:BAapp} displays the results. As we can see, for 3 datasets (auto-mpg, concrete and forest), the aggregated models perform the best. In particular, for the auto-mpg dataset, BA has 3\% improvement over the best base learner. Even for the slump dataset, aggregations still have comparable performance to the best base learner. The two aggregation methods SL and BA have similar performance for all the datasets.

\begin{table*}[h!]
  \centering
\begin{tabular}{c|C{3cm}C{3cm}C{3cm}}
\hhline{====}
   dataset & sample size & \# of predictors & response variable \\
\hline
       auto-mpg & 392 & 8 & mpg \\
       \hline
     concrete & 1030 & 8 & CCS$^{\ast}$ \\
       \hline
       slump & 103 & 7 & slump \\
       \hline
     forest & 517 & 12 & log(1+area) \\
       \hhline{====}
\end{tabular}
  \caption{Descriptions of the four datasets from the UCI repository. CCS: concrete compressive strength.}\label{table:data}
\end{table*}

\begin{table*}[h!]
  \centering
\begin{tabular}{c|C{1cm}C{1cm}C{1cm}C{1cm}C{1cm}C{1cm}C{1cm}C{1cm}C{1cm}}
\hhline{==========}
   dataset & Cart & RF & Lasso & SVM & Ridge & NN & GAM & SL & BA \\
\hline
       auto-mpg & 3.42 & 2.69 & 3.38 & 2.68 & 3.40 & 7.79 & 2.71 & \textbf{2.61} & \textbf{2.61} \\
       \hline
     concrete & 9.40 & 5.35 & 10.51 & 6.65 & 10.50 & 16.64 & 7.95 & \textbf{5.31} & 5.33 \\
       \hline
       slump & 7.60 & \textbf{6.69} & 7.71 & 7.05 & 8.67 & 7.11 & 6.99 & 7.17 & 7.03 \\
       \hline
     forest & .670 & .628 & .612 & .612 & .620 & .613 & .622 & .606 & \textbf{.604} \\
       \hhline{==========}
\end{tabular}
  \caption{RMSE of aggregations for real data applications. All results are based on 10-fold cross-validations.}\label{table:BAapp}
\end{table*}

\section{Proofs of the main results}\label{se:proofs}
Let $K(P,Q)=\int\log(dP/dQ)dP$ be the KL divergence between two probability distributions $P$ and $Q$, and $V(P,Q)=\int|\log(dP/dQ)-K(P,Q)|^2dP$ be a discrepancy measure.

\subsection{Concentration properties of Dirichlet distribution and double Dirichlet distribution}\label{se:pcr}
According to the posterior asymptotic theory developed in \cite{Ghosal2000} (for iid observations, e.g. regression with random design, such as the aggregation problem in section \ref{se:BA}), to ensure a posterior convergence rate of at least $\epsilon_n$, the prior has to put enough mass around $\theta^\ast $ in the sense that \begin{align*}
    & \text{(PC1)} & &\Pi(B(\theta^\ast ,\epsilon_n))\geq e^{-n\epsilon_n^2C},\ \text{with } &&\\ &&
    B(\theta^\ast ,\epsilon)=&\{\theta\in\Theta: K(P_{\theta^\ast }, P_{\theta})\leq \epsilon^2,\ V(P_{\theta^\ast }, P_{\theta})\leq \epsilon^2\},&&
\end{align*}
for some $C>0$. For independent but non-identically distributed (noniid) observations (e.g. regression with fixed design, such as the linear regression problem in section \ref{se:lreg}), where the likelihood takes a product form $P^{(n)}_{\theta}(Y_1,\ldots,Y_n)=\prod_{i=1}^n P_{\theta,i}(Y_i)$, the corresponding prior concentration condition becomes \citep{Ghosal2007}
\begin{align*}
    & \text{(PC2)} & &\Pi(B_n(\theta^\ast ,\epsilon_n))\geq e^{-n\epsilon_n^2C}, \text{ with }&&\\ & &
    B_n(\theta^\ast ,\epsilon)=&\bigg\{\theta\in\Theta: \frac{1}{n}\sum_{i=1}^nK(P_{\theta^\ast ,i}, P_{\theta,i})\leq \epsilon^2,\ \frac{1}{n}\sum_{i=1}^nV(P_{\theta^\ast ,i}, P_{\theta,i})\leq \epsilon^2\bigg\}.&&
\end{align*}
If a (semi-)metric $d_n$, which might depend on $n$, dominates $KL$ and $V$ on $\Theta$, then (PC) is implied by $\Pi(d_n(\theta,\theta^\ast )\leq c\epsilon_n)\geq e^{-n\epsilon_n^2C}$ for some $c>0$. In the aggregation problem with a random design and parameter $\theta=\lambda$, we have $K(P_{\theta^\ast }, P_{\theta})=V(P_{\theta^\ast }, P_{\theta})=\frac{1}{2\sigma^2}||\sum_{j=1}^M(\lambda_j-\lambda_j^\ast )f_j||_Q^2
=\frac{1}{2\sigma^2}(\lambda-\lambda^\ast )^T\Sigma(\lambda-\lambda^\ast )$. Therefore, we can choose $d_n(\theta,\theta^\ast )$ as $d_{\Sigma}(\lambda,\lambda^\ast )=||\Sigma^{1/2}(\lambda-\lambda^\ast )||_2$. In the linear aggregation problem with a fixed design and $\theta=\lambda$, $\sum_{j=1}^nK(P_{\theta^\ast ,i}, P_{\theta,i})=\sum_{j=1}^nV(P_{\theta^\ast ,i}, P_{\theta,i})=\frac{1}{2\sigma^2}||F(\lambda-\lambda^\ast )||_2$, where $P_{\theta,i}(Y)=P_{\lambda}(Y|F^{(i)})$. Therefore, we can choose $d_n(\theta,\theta^\ast )$ as $d_F(\lambda,\lambda^\ast )=||\frac{1}{\sqrt{n}}F(\lambda-\lambda^\ast )||_2$.

For CA and LA, the concentration probabilities can be characterized by those of $\lambda^\ast \in \Lambda$ and $\eta^\ast \in D_{M-1}=\{\eta\in\bbR^M, ||\eta||_1=1\}$.
Therefore, it is important to investigate the concentration properties of the Dirichlet distribution and the double Dirichlet distribution as priors over $\Lambda$ and $D_{M-1}$.
The concentration probabilities $\Pi(d_{\Sigma}(\lambda,\lambda^\ast )\leq c\epsilon)$ and $\Pi(d_{F}(\eta,\eta^\ast )\leq c\epsilon)$ depend on the location of the centers $\lambda^\ast $ and $\eta^\ast $, which are characterized by their geometrical properties, such as sparsity and $l_1$-summability. The next lemma characterizes these concentration probabilities and is of independent interest.

\begin{lemma}\label{le:conprob}
Assume (A1) and (B1).
\begin{enumerate}
  \item[a.] Assume (A2). Under the prior (DA), for any $\gamma\geq1$,
  \begin{align*}
    \Pi(d_{\Sigma}(\lambda,\lambda^\ast )\leq \epsilon)\geq \exp\bigg\{-C\gamma s\log\frac{M}{\epsilon} \bigg\},\quad\text{for some } C>0.
  \end{align*}
  \item[b.] Under the prior (DA), for any $m>0$, any $\lambda\in \Lambda$ and any $\gamma\geq1$,
  \begin{align*}
    &\Pi\bigg(d_{\Sigma}(\lambda,\lambda^\ast )\leq \epsilon+\frac{C}{\sqrt{m}}\bigg)\geq \exp\bigg\{-C\gamma m\log\frac{M}{\epsilon}  \bigg\},\\
    &\Pi(d_{\Sigma}(\lambda,\lambda^\ast )\leq \epsilon)\geq \exp\bigg\{-C\gamma M\log\frac{M}{\epsilon}  \bigg\}, \quad\text{for some } C>0.
  \end{align*}
  \item[c.] Assume (B2a). Under the prior for $\eta$ in (DDG2), for any $\gamma\geq1$,
  \begin{align*}
    \Pi(d_F(\eta,\eta^\ast )\leq \epsilon)\geq \exp\bigg\{-C\gamma s\log\frac{M}{\epsilon} \bigg\},\quad\text{for some } C>0.
  \end{align*}
  \item[d.] Under the prior for $\eta$ in (DDG2), for any $m>0$, any $\eta\in D_{M-1}$ and any $\gamma\geq1$,
  \begin{align*}
    &\Pi\bigg(d_F(\eta,\eta^\ast )\leq \epsilon+\frac{C}{\sqrt{m}}\bigg)\geq \exp\bigg\{-C\gamma m\log\frac{M}{\epsilon}  \bigg\},\\
    &\Pi(d_F(\eta,\eta^\ast )\leq \epsilon)\geq \exp\bigg\{-C\gamma M\log\frac{M}{\epsilon}  \bigg\},\quad\text{for some } C>0.
  \end{align*}
\end{enumerate}
\end{lemma}

\begin{itemize}
  \item The lower bound $\exp\{-C\gamma s\log (M/\epsilon)\}$ in Lemma \ref{le:conprob} can be decomposed into two parts: $\exp\{-C\gamma s\log M\}$ and $\exp\{Cs\log \epsilon\}$. The first part has the same order as $1/{M \choose s}$, one over the total number of ways to choose $s$ indices from $\{1,\ldots,M\}$. The second part is of the same order as $\epsilon^{s}$, the volume of an $\epsilon$-cube in $\bbR^s$. Since usually which $s$ components are nonzero and where the vector composed of these $s$ nonzero components locates in $\bbR^s$ are unknown, this prior lower bound cannot be improved.
  \item  The priors (DA) and (DDG2) do not depend on the sparsity level $s$. As a result, Lemma \ref{le:conprob} suggests that the prior concentration properties hold simultaneously for all $\lambda^\ast $ or $\eta^\ast $ with different $s$ and thus these priors can adapt to an unknown sparsity level.
\end{itemize}

By the first two parts of Lemma \ref{le:conprob}, the following is satisfied for the prior (DA) with $D$ large enough,
\begin{align*}
 &\text{(DA-PC)} & \Pi(d_{\Sigma}(\lambda,\lambda^\ast )\leq \epsilon_n)\geq e^{-n\epsilon_n^2C},\ \text{with } \epsilon_n=\left\{
              \begin{array}{cl}
                D\sqrt{\frac{s\log (M/s)}{n}}, & \text{ if }||\lambda^\ast ||_0=s; \\
                D\sqrt{\frac{M}{n}}, & \text{ if }M\leq \sqrt{n};\\
                D\sqrt[4]{\frac{\log(M/\sqrt{n}+1)}{n}}, & \text{ if }M>\sqrt{n}.
              \end{array}
            \right. &&
\end{align*}
This prior concentration property will play a key role in characterizing the posterior convergence rate of the prior (DA) for Bayesian aggregation.

Based on the prior concentration property of the double Dirichlet distribution
provided in Lemma \ref{le:conprob} part c and part d, we have the corresponding property for the prior (DDG2) by taking into account the prior distribution of $A=||\lambda||_1$.

\begin{corollary}\label{coro:pclr}
Assume (B1).
\begin{enumerate}
\item[a.] Assume (B2a). Under the prior (DDG2), for any $\gamma\geq1$,
  \begin{align*}
    \Pi(d_F(\lambda,\lambda^\ast )\leq \epsilon)\geq \exp\bigg\{-C\gamma s\log\frac{M}{\epsilon} \bigg\},\quad\text{for some } C>0.
  \end{align*}
\item[b.] Assume (B2b). Under the prior (DDG2), for any $m>0$, any $\eta\in D_{M-1}$ and any $\gamma\geq1$,
  \begin{align*}
  &\Pi\bigg(d_F(\lambda,\lambda^\ast )\leq \epsilon+\frac{C}{\sqrt{m}}\bigg)\geq \exp\bigg\{-C\gamma m\log\frac{M}{\epsilon}  \bigg\},\\
    &\Pi(d_F(\lambda,\lambda^\ast )\leq \epsilon)\geq \exp\bigg\{-C\gamma M\log\frac{M}{\epsilon}  \bigg\},\quad\text{for some } C>0.
  \end{align*}
\end{enumerate}
\end{corollary}
Based on the above corollary, we have a similar prior concentration property for the prior (DDG2):
\begin{align*}
 &\text{(DDG2-PC)} & \Pi(d_F(\theta,\theta^\ast )\leq c\epsilon_n)\geq e^{-n\epsilon_n^2C},\ \text{with the same $\epsilon_n$ in (DA-PC)}.&&
\end{align*}

\subsection{Supports of the Dirichlet distribution and the double Dirichlet distribution}
By \cite{Ghosal2000}, a second condition
to ensure the posterior convergence rate of $\theta^\ast \in\Theta$ at least $\epsilon_n$ is that the prior $\Pi$ should put almost all its mass in a sequence of subsets of $\Theta$ that are not too complex. More precisely, one needs to show that there exists a sieve sequence $\{\mathcal{F}_n\}$ such that $\theta^\ast \in \mathcal{F}_n\subset\Theta$, $\Pi(\mathcal{F}_n^c)\leq e^{-n\epsilon_n^2C}$ and $\log N(\epsilon_n,\mathcal{F}_n,d_n)\leq n\epsilon^2_n$ for each $n$, where for a metric space $\mathcal{F}$ associated with a (semi-)metric $d$, $N(\epsilon,\mathcal{F},d)$ denotes the minimal number of $d$-balls with radii $\epsilon$ that are needed to cover $\mathcal{F}$.

For the priors (DA) and (DDG2), the probability of the space of all $s$-sparse vectors is zero. We consider the approximate $s$-sparse vector space $\mathcal{F}^{\Lambda}_{s,\epsilon}$ defined in Section \ref{se:cdd} for CA. For LA, we define $\mathcal{F}^{D}_{B,s,\epsilon}=\{\theta=A\eta:\eta\in D_{M-1}, \sum_{j=s+1}^{M}|\eta_{(j)}|\leq B^{-1}\epsilon; \ 0\leq A\leq B \}$,
where $|\eta_{(1)}|\geq\cdots\geq |\eta_{(M)}|$ is the ordered sequence of $\eta_1,\ldots,\eta_M$ according to their absolute values.

The following lemma characterizes the complexity of $\Lambda$, $D^B_{M-1}=\{A\eta:\eta\in D_{M-1};\ 0\leq A\leq B\}$, $\mathcal{F}^{\Lambda}_{s,\epsilon}$ and $\mathcal{F}^{D}_{B,s,\epsilon}$ in terms of their covering numbers.

\begin{lemma}\label{le:coen}
Assume (A1) and (B1).
\begin{enumerate}
  \item[a.] For any $\epsilon\in(0,1)$, integer $s>0$ and $B>0$, we have
\begin{align*}
  &\log N(\epsilon, \mathcal{F}^{\Lambda}_{s,\epsilon}, d_{\Sigma})\lesssim s\log\frac{M}{\epsilon},\\
  &\log N(\epsilon, \mathcal{F}^{D}_{B,s,\epsilon}, d_F)\lesssim s\log\frac{M}{\epsilon}+s\log B.
\end{align*}
  \item[b.] For any $\epsilon\in(0,1)$ and integer $m>0$, we have
\begin{align*}
    &\log N(C/\sqrt{m}, \Lambda, d_{\Sigma})\lesssim m\log M,\\
    &\log N(\epsilon, \Lambda, d_{\Sigma})\lesssim M\log\frac{M}{\epsilon},\\
    &\log N(CB/\sqrt{m}, BD_{M-1}, d_F)\lesssim m\log M,\\
    &\log N(B\epsilon, BD_{M-1}, d_F)\lesssim M\log\frac{M}{\epsilon}.
\end{align*}
\end{enumerate}
\end{lemma}

The next lemma provides upper bounds to the complementary prior probabilities of $\mathcal{F}^{\Lambda}_{s,\epsilon}$ and $\mathcal{F}^{D}_{B,s,\epsilon}$. The proof utilizes the connection between the Dirichlet distribution and the stick-breaking representation of the Dirichlet processes \citep{sethuraman1994}.

\begin{lemma}\label{le:cpp}
\begin{enumerate}
  \item[a.] For any $\epsilon\in(0,1)$, under the prior (DA) with $\gamma>1$, we have
\begin{align*}
    \Pi(\lambda\notin \mathcal{F}^{\Lambda}_{s,\epsilon})\lesssim \exp\bigg\{-Cs(\gamma-1)\log \frac{M}{\epsilon}\bigg\}.
\end{align*}
  \item[b.] For any $\epsilon\in(0,1)$, under the prior (DDG2) with $\gamma>1$, we have
\begin{align*}
    \Pi(\theta\notin \mathcal{F}^{D}_{B,s,\epsilon})\lesssim
     \exp\bigg\{-Cs(\gamma-1)\log \frac{M}{\epsilon}-Cs\log B\bigg\}+\exp\{-CB\}.
\end{align*}
\end{enumerate}
\end{lemma}

\subsection{Test construction}
For CA, we use the notation $P_{\lambda,Q}$ to denote the joint distribution of $(X,Y)$, whenever $X\sim Q$ and $Y|X\sim N(\sum_{j=1}^M\lambda_jf_j(X),\sigma^2)$ for any $\lambda\in \Lambda$ and $E_{\lambda,Q}$ the expectation with respect to $P_{\lambda,Q}$. Use $P_{\lambda,Q}^{(n)}$ to denote the $n$-fold convolution of $P_{\lambda,Q}$. Let $X^n=(X_1,\ldots,X_n)$ and $Y^n=(Y_1,\ldots,X_n)$ be $n$ copies of $X$ and $Y$.  Recall that $f_0$ is the true regression function that generates the data. We use $P_{0,Q}$ to denote the corresponding true distribution of $Y$. For LA, we use $P_{\lambda}$ to denote the distribution of $Y$, whenever $Y\sim N(F\lambda,\sigma^2I_n)$ and $E_{\lambda}$ the expectation with respect to $P_{\lambda}$.

For both aggregation problems, we use the ``M-open" view where $f_0$ might not necessarily belong to $\mathcal{F}^{\Lambda}$ or $\mathcal{F}^{\bbR^M}$. We apply the result in \cite{Kleijn2006} to construct a test under misspecification for CA with random design and explicitly construct a test under misspecification for LA with fixed design. Note that the results in \cite{Kleijn2006} only apply for random-designed models. For LA with fixed design, we construct a test based on concentration inequalities for Gaussian random variables.

\begin{lemma}\label{le:test}
Assume (A3).
\begin{enumerate}
  \item[a.] Assume that $f^\ast =\sum_{j=1}^M\lambda_j^\ast f_j$ satisfies $E_Q(f-f^\ast )(f^\ast -f_0)=0$ for every $f\in\mathcal{F}^{\Lambda}$. Then there exist $C>0$ and a measurable function $\phi_n$ of $X^n$ and $Y^n$, such that for any other vector $\lambda_2\in \Lambda$,
      \begin{align*}
    &P_{0,Q}^{(n)}\phi_n(X^n,Y^n)\leq \ \exp\big\{-Cnd^2_{\Sigma}(\lambda_2,\lambda^\ast )\big\}\\
    \sup_{\lambda\in \Lambda:\ d_{\Sigma}(\lambda,\lambda_2)<\frac{1}{4}d_F(\lambda^\ast ,\lambda_2)}
    &P_{\lambda,Q}^{(n)}(1-\phi_n(X^n,Y^n))\leq \ \exp\big\{-Cnd^2_{\Sigma}(\lambda_2,\lambda^\ast )\big\}.
\end{align*}
  \item[b.] Assume that $\lambda^\ast\in\bbR^d$ satisfies $F^T(F\lambda^\ast-F_0)=0$ for every $\lambda\in\bbR^d$. Then there exists a measurable function $\phi_n$ of $Y$ and some $C>0$, such that for any other $\lambda_2\in\bbR^d$,
  \begin{align*}
    &P_{0}\phi_n(Y)\leq \ \exp\big\{-Cnd^2_F(\lambda_2,\lambda^\ast)\big\}\\
    \sup_{\lambda\in \bbR^M:\ d_F(\lambda,\lambda_2)<\frac{1}{4}d_F(\lambda^\ast,\lambda_2)}
    &P_{\lambda}(1-\phi_n(Y))\leq \ \exp\big\{-Cnd^2_X(\lambda_2,\lambda^\ast)\big\}.
\end{align*}
\end{enumerate}
\end{lemma}

\begin{itemize}
  \item As we discussed in the remark in section \ref{se:pcrba}, in order to apply \cite{Kleijn2006} for Gaussian regression with random design, we need the mean function to be uniformly bounded. For the convex aggregation space $\mathcal{F}^{\Lambda}$, this uniformly bounded condition is implied by (A3). For the linear regression with fixed design, we do not need the uniformly bounded condition. This property ensures that the type I and type II errors in b do not deteriorate as $||\lambda_2||_1$ grows, which plays a critical role in showing that the posterior probability of $\{A>C A^\ast \}$ converges to zero in probability for $C$ sufficiently large, where $A=||\lambda||_1$ and $A^\ast =||\lambda^\ast ||_1$. Similarly, if we consider CA with a fixed design, then only an assumption like (B1) on the design points is needed.
  \item The assumption on $f^\ast $ in CA is equivalent to that $f^\ast $ is the minimizer over $f\in\mathcal{F}^{\Lambda}$ of $||f-f_0||^2_Q$, which is proportional to the expectation of the KL divergence between two normal distributions with mean functions $f_0(X)$ and $f(X)$ with $X\sim Q$. Therefore, $f^\ast $ is the best $L_2(Q)$-approximation of $f_0$ within the aggregation space $\mathcal{F}^{\Lambda}$ and the lemma suggests that the likelihood function under $f^\ast $ tends to be exponentially larger than other functions in $\mathcal{F}^{\Lambda}$. Similarly, the condition on $\lambda^\ast$ in LA is equivalent to that $\lambda^\ast$ is the minimizer over $\lambda\in\bbR^d$ of $||F\lambda-F_0||_2^2$, which is proportional to the KL divergence between two multivariate normal distributions with mean vectors $F\lambda$ and $F_0$.
\end{itemize}

\subsection{Proof of Theorem \ref{thm:BA}}
The proof follows similar steps as the proof of Theorem 2.1 in \cite{Ghosal2000}. The difference is that we consider the misspecified framework where the asymptotic limit of the posterior distribution of $f$ is $f^\ast $ instead of the true underlying regression function $f_0$. As a result, we need to apply the test condition in Lemma \ref{le:test} part a in the model misspecified framework. We provide a sketched proof as follows.

Let $\epsilon_n$ be given by (DA-PC) and $\Pi^B(\lambda)=\Pi\big(\lambda|B(\lambda^\ast ,\epsilon_n)\big)$ with $B(\lambda^\ast ,\epsilon_n)$ defined in (PC1).
By Jensen's inequality applied to the logarithm,
\begin{align*}
    \log\int_{B(\lambda^\ast ,\epsilon_n)}\prod_{i=1}^n\frac{P_{\lambda,Q}}{P_{0,Q}}(X_i,Y_i)d\Pi^{B}(\lambda)
    \geq \sum_{i=1}^n\int_{B(\lambda^\ast ,\epsilon_n)}\log\frac{P_{\lambda,Q}}{P_{0,Q}}(X_i,Y_i)d\Pi^{B}(\lambda).
\end{align*}
By the definition of $B(\lambda^\ast ,\epsilon_n)$ and an application of Chebyshev's inequality, we have that for any $C>0$,
\begin{align*}
    &P_0\bigg\{\sum_{i=1}^n\int_{B(\lambda^\ast ,\epsilon_n)}\bigg(\log\frac{P_{\lambda,Q}}{P_{0,Q}}(X_i,Y_i)
    +K(P_{\lambda_0,Q}, P_{\lambda,Q})\bigg)d\Pi^{B}(\lambda)\\
    &\qquad\qquad\qquad\qquad\qquad\leq -(1+C)n\epsilon_n^2+n\int_{B(\lambda^\ast ,\epsilon_n)}K(P_{\lambda_0,Q}, P_{\lambda,Q})d\Pi^B(\lambda)\bigg\}\\
    \leq\
    &P_0\bigg\{\sum_{i=1}^n\int_{B(\lambda^\ast ,\epsilon_n)}\bigg(\log\frac{P_{\lambda,Q}}{P_{0,Q}}(X_i,Y_i)
    +K(P_{\lambda_0,Q}, P_{\lambda,Q})\bigg)d\Pi^{B}(\lambda)\leq -Cn\epsilon_n^2\bigg\}\\
    \leq&\ \frac{n \int_{B(\lambda^\ast ,\epsilon_n)}V(P_{\lambda_0,Q}, P_{\lambda,Q})d\Pi^B(\lambda)}{(Cn\epsilon_n^2)^2}\leq \frac{1}{C^2n\epsilon_n^2}\to 0, \text{ as }n\to \infty.
\end{align*}
Combining the above two yields that on some set $A_n$ with $P_0$-probability converging to one,
\begin{align}\label{eq:icon}
    \int_{B(\lambda^\ast ,\epsilon_n)}\prod_{i=1}^n\frac{P_{\lambda,Q}}{P_{0,Q}}(X_i,Y_i)d\Pi(\lambda)
    \geq \exp(-(1+C)n\epsilon_n^2)\Pi(B(\lambda^\ast ,\epsilon_n))\geq \exp(-C_0n\epsilon_n^2),
\end{align}
for some $C_0>0$, where we have used the fact that $\Pi(B(\lambda^\ast ,\epsilon_n))\geq \Pi(d_{\Sigma}(\lambda,\lambda^\ast )\leq C\epsilon_n)\geq \exp(-Cn\epsilon_n^2)$ for some $C>0$.

Let $\mathcal{F}_n=\mathcal{F}_{as,\epsilon_n}^\lambda$ for some $a>0$ if (A2) holds and otherwise $\mathcal{F}_n=\Lambda$. Then by Lemma \ref{le:coen} part a and Lemma \ref{le:cpp} part a, for some constants $C_1>0$ and $C_2>0$,
\begin{align}\label{eq:sievec}
    \log N(\epsilon_n,\mathcal{F}_n,d_{\Sigma})\leq C_1n\epsilon_n^2,\quad \Pi(\lambda\notin \mathcal{F}_n)\leq\exp(-C_2n\epsilon_n^2).
\end{align}
Because $C_2$ is increasing with the $a$ in the definition of $\mathcal{F}_n$, we can assume $C_2>C_0+1$ by properly selecting an $a$.

For some $D_0>0$ sufficiently large, let $\lambda_1^\ast ,\ldots,\lambda_J^\ast \in \mathcal{F}_n-\{\lambda:d_{\Sigma}(\lambda,\lambda^\ast )\leq 4D_0\epsilon_n\}$ with $|J|\leq \exp(C_1n\epsilon_n^2)$ be $J$ points that form an $D_0\epsilon_n$-covering net of $\mathcal{F}_n-\{\lambda:d_{\Sigma}(\lambda,\lambda^\ast )\leq 4D_0\epsilon_n\}$. Let $\phi_{j,n}$ be the corresponding test function provided by Lemma \ref{le:test} part a with $\lambda_2=\lambda_j^\ast $ for $j=1,\ldots,J$. Set $\phi_n=\max_{j}\phi_{j,n}$. Since $d_{\Sigma}(\lambda_j^\ast ,\lambda^\ast )\geq 4D_0\epsilon_n$ for any $j$, we obtain
\begin{align}\label{eq:pcrtest1}
    P_{0,Q}^{(n)}\phi_n\leq \sum_{j=1}^JP_{0,Q}^{(n)}\phi_{j,n}\leq |J|\exp(-C16D_0^2n\epsilon_n^2)\leq\exp(-C_3n\epsilon_n^2),
\end{align}
where $C_3=16CD_0^2-1>0$ for $D_0$ large enough.
For any $\lambda\in\mathcal{F}_n-\{\lambda:d_{\Sigma}(\lambda,\lambda^\ast )\leq 4 D_0\epsilon_n\}$, by the design, there
exists a $j_0$ such that $d_{\Sigma}(\lambda_{j_0}^\ast ,\lambda)\leq D_0\epsilon_n$. This implies that $d_{\Sigma}(\lambda_{j_0}^\ast ,\lambda^\ast )\geq 4D_0\epsilon_n\geq 4d_{\Sigma}(\lambda_{j_0}^\ast ,\lambda)$, therefore
\begin{align}
  & \sup_{\lambda\in \mathcal{F}_n:\ d_{\Sigma}(\lambda,\lambda^\ast )\geq 4D_0\epsilon_n} P_{\lambda,Q}^{(n)}\phi_n\nonumber\\
  \leq &\ \min_{j}\sup_{\lambda\in \Lambda:\ d_{\Sigma}(\lambda,\lambda_{j}^\ast )<\frac{1}{4}d_{\Sigma}(\lambda^\ast ,\lambda_{j}^\ast )}
    P_{\lambda,Q}^{(n)}(1-\phi_n)\leq \ \exp\big\{-C_4n\epsilon_n^2\big\},\label{eq:pcrtest2}
\end{align}
with $C_4=16CD_0^2>C_0+1$ with $D_0$ sufficiently large. With $D=4D_0$, we have
\begin{align}
   & E_{0,Q}\Pi(d_{\Sigma}(\lambda,\lambda^\ast )
    \geq D\epsilon_n|X_1,Y_1,\ldots,X_n,Y_n)I(A_n)\nonumber\\
    \leq &\ P_{0,Q}^{(n)}\phi_n+E_{0,Q}\Pi(d_{\Sigma}(\lambda,\lambda^\ast )\geq D\epsilon_n|X_1,Y_1,\ldots,X_n,Y_n)I(A_n)(1-\phi_n).\label{eq:pcrkey1}
\end{align}
By \eqref{eq:icon}, \eqref{eq:sievec} and \eqref{eq:pcrtest2}, we have
\begin{align}
    &E_{0,Q}\Pi(d_{\Sigma}(\lambda,\lambda^\ast )\geq D\epsilon_n|X_1,Y_1,\ldots,X_n,Y_n)I(A_n)(1-\phi_n)\nonumber\\
    \leq &\ P_{0,Q}^{(n)}(1-\phi_n)I(A_n)\frac{\int_{\lambda\in\mathcal{F}_n:d_{\Sigma}(\lambda,\lambda^\ast )\geq D\epsilon_n}\prod_{i=1}^n\frac{P_{\lambda,Q}}{P_{0,Q}}(X_i,Y_i)d\Pi(\lambda)}
    {\int_{B(\lambda^\ast ,\epsilon_n)}\prod_{i=1}^n\frac{P_{\lambda,Q}}{P_{0,Q}}(X_i,Y_i)d\Pi(\lambda)}
    \nonumber\\
    &\qquad\qquad\qquad+P_{0,Q}^{(n)}I(A_n)\frac{\int_{\lambda\notin\mathcal{F}_n}\prod_{i=1}^n\frac{P_{\lambda,Q}}{P_{0,Q}}(X_i,Y_i)d\Pi(\lambda)}
    {\int_{B(\lambda^\ast ,\epsilon_n)}\prod_{i=1}^n\frac{P_{\lambda,Q}}{P_{0,Q}}(X_i,Y_i)d\Pi(\lambda)}\nonumber\\
    \leq &\ \exp(C_0n\epsilon_n^2)\sup_{\lambda\in \mathcal{F}_n:\ d_{\Sigma}(\lambda,\lambda^\ast )\geq 4D_0\epsilon_n} P_{\lambda,Q}^{(n)}\phi_n
    +\exp(C_0n\epsilon_n^2)\Pi(\lambda\notin\mathcal{F}_n)
    \leq \ 2\exp(-n\epsilon_n^2).\label{eq:pcrproof}
\end{align}
Combining the above with \eqref{eq:pcrtest1}, \eqref{eq:pcrkey1}, and the fact that $E_{0,Q}I(A_n^c)\to 0$ as $n\to\infty$, Theorem \ref{thm:BA} can be proved.

\subsection{Proof of Theorem \ref{thm:LR}}
For the sparse case where (B2a) is satisfied, we construct the sieve by $\mathcal{F}_n=\mathcal{F}^{D}_{bn\epsilon_n^2,as,\epsilon_n}$ with the $\epsilon_n$ given in (DDG2-PC), where $a>0$, $b>0$ are sufficiently large constants. Then by Lemma \ref{le:coen} part a and Lemma \ref{le:cpp} part b, we have
\begin{align}\label{eq:sievelr}
    \log N(\epsilon_n, \mathcal{F}_n,d_F)\leq C_1n\epsilon_n^2,\quad
    \Pi(\lambda\notin\mathcal{F}_n)\leq& \Pi(-C_2n\epsilon_n^2),
\end{align}
where $C_2$ is increasing with $a$ and $b$. The rest of the proof is similar
to the proof of Theorem \ref{thm:BA} with the help of \eqref{eq:sievelr}, Corollary \ref{le:conprob} part b and Lemma \ref{le:test} part b.

Next, we consider the dense case where (B2b) and (B3) are satisfied.
By the second half of Lemma \ref{le:coen} part b, the approximation accuracy of $BD_{M-1}$ degrades linearly in $B$. Therefore, in order to construct a sieve such that \eqref{eq:sievelr} is satisfied with the $\epsilon_n$ given in (DDG2-PC), we need to show that $E_0\Pi(A\leq KA^\ast |Y)\to 0$ as $n\to\infty$ with some constant $K>0$. Then by conditioning on the event $\{A\leq KA^\ast \}$, we can choose $\mathcal{F}_n=BD_{M-1}$ with $B=KA^\ast $, which does not increase with $n$, and \ref{eq:sievelr} will be satisfied. As long as \ref{eq:sievelr} is true, the rest of the proof will be similar to the sparse case.

We only prove that $E_0\Pi(A\leq KA^\ast |Y)\to 0$ as $n\to\infty$ here. By (B1) and (B3), for any $\eta\in D_{M-1}$ and $A>0$, $d_F(A\eta,A^\ast \eta^\ast )\geq \kappa_0A-\kappa A^\ast $. As a result, we can choose $K$ large enough so that $d_F(A\eta,A^\ast \eta^\ast )\geq 4$ for all $A\geq KA^\ast $ and all $\eta\in D_{M-1}$. Therefore, by Lemma \ref{le:test} part b, for any $\lambda_2=A_2\eta_2$ with $A_2>KA^\ast $ and $\eta_2\in D_{M-1}$, there exists a test $\phi_n$ such that
\begin{align*}
    &P_{\lambda^\ast }\phi_n(Y)\leq \ \exp\big\{-Cn\big\}\\
    \sup_{\lambda\in \bbR^M:\ d_F(\lambda,\lambda_2)<\frac{1}{4}d_F(\lambda^\ast ,\lambda_2)}
    &P_{\lambda}(1-\phi_n(Y))\leq \ \exp\big\{-Cn\big\}.
\end{align*}
By choosing $K$ large enough, we can assume that $\kappa_0KA^\ast/8 >\kappa+\kappa A^\ast/4 $. For any  $\lambda=A\eta$ satisfying $d_F(\eta,\eta_2)\leq \kappa_0/8$ and $|A-A_2|\leq 1$, by (B1) and $A_2>KA^\ast $ we have
\begin{align*}
    &d_F(\lambda,\lambda_2)\leq d_F(A\eta,A_2\eta)+d_F(A_2\eta,A_2\eta_2)\leq \kappa+\frac{1}{8}\kappa_0 A_2\\
    \leq &\frac{1}{4}(\kappa_0A_2-\kappa A^\ast )\leq\frac{1}{4}d_F(\lambda^\ast ,\lambda_2).
\end{align*}
Combining the above, we have that for any $\lambda_2=A_2\eta_2$ with $A_2>KA^\ast $ and $\eta_2\in D_{M-1}$,
\begin{align*}
    &P_{\lambda^\ast }\phi_n(Y)\leq \ \exp\big\{-Cn\big\}\\
    \sup_{|A-A_2|\leq1, d_F(\eta,\eta_2)\leq\kappa_0/8}
    &P_{\lambda}(1-\phi_n(Y))\leq \ \exp\big\{-Cn\big\}.
\end{align*}
Let $A_1^\ast ,\ldots,A_{J_1}^\ast $ be a $1$-covering net of the interval $[KA^\ast ,Cn\epsilon_n^2]$ with $J_1\leq Cn\epsilon_n^2$ and $\eta_1^\ast ,\ldots,\eta_{J_2}^\ast $ be a $\kappa_0/8$-covering net of $D_{M-1}$ with $\log J_2\leq C n\epsilon_n^2$ (by Lemma \ref{le:coen} part b with $B=1$). Let $\phi_{j}$ ($j=1,\ldots,J_1J_2$) be the corresponding tests associated with each combination of $(A_s^\ast ,\eta_t^\ast )$ for $s=1,\ldots,J_1$ and $t=1,\ldots,J_2$. Let $\phi_n=\max_{j}\phi_j$. Then for $n$ large enough,
\begin{equation}\label{eq:testA}
    \begin{aligned}
    &P_{\lambda^\ast }\phi_n(Y)\leq \ \exp\big\{\log(n\epsilon_n^2)+Cn\epsilon_n^2-Cn\big\}\leq \exp\big\{-Cn\}\\
    \sup_{\lambda=A\eta:A\in[KA^\ast ,Cn\epsilon_n^2], \eta\in D_{M-1}}
    &P_{\lambda}(1-\phi_n(Y))\leq \ \exp\big\{-Cn\big\}.
\end{aligned}
\end{equation}
Moreover, because $A\sim$ Ga$(a_0,b_0)$, we have
\begin{align}\label{eq:cpA}
    \Pi(\lambda\notin Cn\epsilon_n^2 D_{M-1})\leq \Pi(A>Cn\epsilon_n^2)\leq\exp\{-Cn\epsilon_n^2\}.
\end{align}
Combining \eqref{eq:testA} and \eqref{eq:cpA}, we can prove that $E_0\Pi(A\leq KA^\ast |Y)\to 0$ as $n\to\infty$ by the same arguments as in \eqref{eq:pcrproof}.

\section{Technical Proofs}

\subsection{Proof of Lemma \ref{le:conprob}}
The following lemma suggests that for any $m>0$, each point in $\Lambda$ or $D_{M-1}$ can be approximated by an $m$-sparse point in the same space with error at most $\sqrt{2\kappa/m}$.

\begin{lemma}\label{le:prep2}
Fix an integer $m\geq 1$. Assume (A1) and (B1).
\begin{enumerate}
  \item[a.] For any $\lambda^\ast \in \Lambda$, there exists a $\bar{\lambda}\in \Lambda$, such that $||\bar{\lambda}||_0\leq m$ and $d_{\Sigma}(\bar{\lambda},\lambda^\ast )\leq \sqrt{\frac{2\kappa}{m}}$.
  \item[b.] For any $\eta^\ast \in D_{M-1}$, there exists an $\bar{\eta}\in D_{M-1}$, such that $||\bar{\eta}||_0\leq m$ and $d_F(\bar{\eta},\eta^\ast )\leq \sqrt{\frac{2\kappa}{m}}$.
\end{enumerate}
\end{lemma}
\begin{proof}
(Proof of a)
Consider a random variable $J\in\{1,\ldots,M\}$ with probability distribution $P(J=j)=\lambda^\ast _j$, $j=1,\ldots,M$. Let $J_1,\ldots,J_m$ be $m$ iid copies of $J$ and $n_j$ be the number of $i\in\{1,\ldots,n\}$ such that $(J_i=j)$. Then $(n_1,\ldots,n_M)\sim$ MN$(m,(\lambda_1^\ast ,\ldots,\lambda_M^\ast )$, where MN denotes the multinomial distribution. Let $V=(n_1/m,\ldots,n_M/m)\in \Lambda$. Then the expectation $E[V]$ of the vector $V$ is $\lambda^\ast $. Therefore, we have
\begin{align*}
    Ed_{\Sigma}^2(V,\lambda^\ast )=&\sum_{j,k=1}^M\Sigma_{jk}E\bigg(\frac{n_j}{m}-\lambda_j^\ast \bigg)\bigg(\frac{n_k}{m}-\lambda_k^\ast \bigg)\\
    =&\ \frac{1}{m}\sum_{j=1}^M\Sigma_{jj}\lambda_j^\ast (1-\lambda_j^\ast )-
    \frac{2}{m}\sum_{1\leq j<k\leq M}\Sigma_{jk}\lambda_j^\ast \lambda_k^\ast \\
    \leq &\ \frac{\kappa}{m}\sum_{j=1}^M\lambda_j^\ast (1-\lambda_j^\ast )+\frac{2\kappa}{m}\sum_{1\leq j<k\leq M}\lambda_j^\ast \lambda_k^\ast \ \leq\  \frac{2\kappa}{m},
\end{align*}
where we have used (A1), the fact that $|\Sigma_{jk}|\leq\Sigma_{jj}^{1/2}\Sigma_{kk}^{1/2}$ and $\sum_{j=1}^M\lambda_j^\ast =1$.
Since the expectation of $d_{\Sigma}^2(V,\lambda^\ast )$ is less than or equal to $2\kappa/m$, there always exists a $\bar{\lambda}\in \Lambda$ such that
$d_{\Sigma}(\bar{\lambda},\lambda^\ast )\leq \sqrt{2\kappa/m}$.

(Proof of b) The proof is similar to that of a. Now we define $J\in\{1,\ldots,M\}$ as a random variable with probability distribution $P(J=j)=|\eta^\ast _j|$, $j=1,\ldots,M$ and let $V=(\text{sgn}(\eta_1^\ast )n_1/m$, $\ldots$, sgn$(\eta_M^\ast )n_M/m)\in D_{M-1}$. The rest follows the same line as part a. under assumption (B1).
\end{proof}

Now, we can proceed to prove Lemma \ref{le:conprob}.

(Proof of a) Without loss of generality, we may assume that the index set of all nonzero components of $\lambda^\ast $ is $S_0=\{1,2,\ldots,s-1,M\}$.
Since $\sup_j\Sigma_{jj}\leq \kappa$ and $|\Sigma_{jk}|\leq\Sigma_{jj}^{1/2}\Sigma_{kk}^{1/2}\leq\kappa$,
\begin{align*}
    d_{\Sigma}(\lambda,\lambda^\ast )=\sum_{j,k=1}^M\Sigma_{jk}(\lambda_j-\lambda_j^\ast )(\lambda_k-\lambda_k^\ast )
    \leq \kappa||\lambda-\lambda^\ast ||_1^2.
\end{align*}
Therefore, for any $\epsilon>0$, $\{||\lambda-\lambda^\ast ||_1\leq \kappa^{-1/2}\epsilon\}\subset\{d_{\Sigma}(\lambda,\lambda^\ast )\leq\epsilon\}$. Since $|\lambda_M-\lambda_M^\ast |\leq\sum_{j=1}^{M-1}|\lambda_j-\lambda^\ast _j|$, for
$\delta_1=\kappa^{-1/2}\epsilon/(4M-4s)$ and $\delta_0=\kappa^{-1/2}\epsilon/(4s)$, we have
\begin{align*}
   \Lambda_{\epsilon}=\big\{\lambda\in \Lambda:\lambda_j\in (0,\delta_1], j\in S_0^c;\ |\lambda_j-\lambda^\ast _j|\leq \delta_0, j\in S_0-\{M\}\big\}\subset\{||\lambda-\lambda^\ast ||_1\leq\kappa^{-1/2}\epsilon\}.
\end{align*}

Combining the above conclusions yields
\begin{align*}
    &\Pi(d_{\Sigma}(\lambda,\lambda^\ast )\leq \epsilon)\
    \geq \ \Pi(\Lambda_{\epsilon})\\
    =&\int_{\Lambda_{\epsilon}}\frac{\Gamma(\alpha/M^{\gamma-1})}{\Gamma^M(\alpha/M^{\gamma})}
\prod_{j=1}^{M-1}\lambda_j^{\alpha/M^{\gamma}-1}
\bigg(1-\sum_{j=1}^{M-1}\lambda_j\bigg)^{\alpha/M^{\gamma}-1}d\lambda_1\cdots d\lambda_{M-1},
\end{align*}
where $\Gamma(\cdot)$ denotes the gamma function. By the facts that $\Gamma(x)\Gamma(1-x)=\pi/\sin(\pi x)$ for $x\in(0,1)$ and $c\triangleq\Gamma'(1)$ is finite, we have $\{x\Gamma(x)\}^{-1}=1-cx+O(x^2)$ for $x\in(0,1/2)$. Combining this with the fact that $\lambda_j\leq1$, we have
\begin{align*}
    \Pi(d_{\Sigma}(\lambda,\lambda^\ast )\leq \epsilon)
    \geq &\ \frac{\Gamma(\alpha/M^{\gamma-1})}{\Gamma^M(\alpha/M^{\gamma})}\bigg\{\prod_{j\in S_0-\{p\}}\int_{\min\{0,\lambda_j-\delta_0\}}^{\max\{1,\lambda_j+\delta_0\}}\lambda_j^{\alpha/M^{\gamma}-1}d\lambda_j\bigg\}
\bigg\{\prod_{j\in S_0^c}\int_{0}^{\delta_1}\lambda_j^{\alpha/M^{\gamma}-1}d\lambda_j\bigg\}\\
\gtrsim &\ \alpha^{-1}M^{\gamma-1}\alpha^{M}M^{-\gamma M} \delta_0^{s-1} \big(\alpha^{-1}M^{\gamma}\delta_1^{\alpha/M^{\gamma}}\big)^{M-s}\\
\gtrsim &\ \alpha^{s-1} M^{-\gamma (s-1)-1}\bigg(\frac{\epsilon}{s}\bigg)^{s-1}\bigg(\frac{\epsilon}{M-s}\bigg)^{\alpha M^{-(\gamma-1)}(1-s/M)}\\
\gtrsim &\ \exp\bigg\{-C\gamma s\log M-Cs\log\frac{s}{\epsilon}\bigg\}
\ \gtrsim\ \exp\bigg\{-C\gamma s\log\frac{M}{\epsilon}\bigg\},
\end{align*}
where we have used the assumption $\gamma\geq1$ and the fact $s\leq M$.

(Proof of b) For any integer $m>0$, let $\bar{\lambda}$ be the $m$-sparse approximation of $\lambda^\ast $ provided in Lemma \ref{le:prep2} part a. Then $d_{\Sigma}(\bar{\lambda},\lambda^\ast )\leq Cm^{1/2}$.
By the conclusion of Lemma \ref{le:conprob} part a, we have
\begin{align*}
    \Pi(d_{\Sigma}(\lambda,\bar{\lambda})\leq \epsilon)\gtrsim \exp\bigg\{-C\gamma m\log \frac{M}{\epsilon}\bigg\}.
\end{align*}
Therefore, by the triangle inequality, we have
\begin{align*}
    \Pi\bigg(d_{\Sigma}(\lambda,\lambda^\ast )\leq \epsilon+\frac{C}{\sqrt{m}}\bigg)\gtrsim \exp\bigg\{-C\gamma m\log \frac{M}{\epsilon}\bigg\}.
\end{align*}

(Proof of c) For the double Dirichlet distribution, the prior mass allocated to each orthant of $\bbR^M$ is $2^{-M}$. A direct application of part a will result a lower bound of order $e^{-CM}$, which is too small compare to our conclusion. Therefore, we need to adapt the proof of part a.

Let $S_0=\{1,2,\ldots,s-1,M\}$ be the index set of all nonzero components of $\eta^\ast $. Similar to the proof of part a, with the same $\delta_1$ and $\delta_0$  we define
\begin{align*}
   \Omega_{\epsilon}=\big\{\eta\in D_{M-1}:|\eta_j|\leq \delta_1, j\in S_0^c;\ |\eta_j-\eta^\ast _j|\leq \delta_0, j\in S_0-\{M\}\big\}.
\end{align*}
Similarly, it can be shown that $\Omega_{\epsilon}\subset\{d_F(\eta,\eta^\ast )\leq\epsilon\}$. So by the fact that $|\eta_j|\leq 1$, we have
\begin{align*}
    \Pi(d_F(\eta,\eta^\ast )\leq \epsilon)
    \geq &\ \frac{1}{2^M}\frac{\Gamma(\alpha/M^{\gamma-1})}{\Gamma^M(\alpha/M^{\gamma})}\bigg\{\prod_{j\in S_0-\{p\}}\int_{\eta_j-\delta_0}^{\eta_j+\delta_0}|\eta_j|^{\alpha/M^{\gamma}-1}d\eta_j\bigg\}
\bigg\{\prod_{j\in S_0^c}\int_{-\delta_1}^{\delta_1}|\eta_j|^{\alpha/M^{\gamma}-1}d\eta_j\bigg\}\\
\gtrsim &\ \frac{1}{2^M}\alpha^{-1}M^{\gamma-1}\alpha^{M}M^{-\gamma M} (2\delta_0)^{s-1} \big(2\alpha^{-1}M^{\gamma}\delta_1^{\alpha/M^{\gamma}}\big)^{M-s}\\
\gtrsim &\ \alpha^{s-1} M^{-\gamma (s-1)-1}\bigg(\frac{\epsilon}{s}\bigg)^{s-1}\bigg(\frac{\epsilon}{M-s}\bigg)^{\alpha M^{-(\gamma-1)}(1-s/M)}\\
\gtrsim &\ \exp\bigg\{-C\gamma s\log\frac{M}{\epsilon}\bigg\}.
\end{align*}
As we can seen, now each $\eta_j$ contributes an additional factor of $2$ to the prior concentration probability comparing to that of $\lambda_j$ in the proof of part a. This additional factor compensates for the $2^{-M}$ factor in the normalizing constant of the double Dirichlet distribution.

(Proof of d) The proof is similar to that of part b by instead combining the proof of part c and Lemma \ref{le:prep2} part b. Therefore, we omit the proof here.

\subsection{Proof of Corollary \ref{coro:pclr}}
By the triangle inequality and assumption (B1), we have
\begin{align*}
    d_F(\lambda,\lambda^\ast )\leq d_F(A\eta, A^\ast \eta)+d_F(A^\ast \eta, A^\ast \eta^\ast )
    \leq \kappa |A-A^\ast |+A^\ast d_F(\eta,\eta^\ast ).
\end{align*}
As a result, $\{|A-A^\ast |\leq \kappa^{-1}\epsilon;\ d_F(\eta,\eta^\ast )\leq (A^\ast )^{-1}\epsilon\}\subset\{d_F(\lambda,\lambda^\ast )\leq 2\epsilon\}$ and
\begin{align*}
    \Pi(d_F(\lambda,\lambda^\ast )\leq \epsilon)\geq \Pi(|A-A^\ast |\leq C\epsilon)\cdot
    \Pi(d_F(\eta,\eta^\ast )\leq C\epsilon).
\end{align*}
Since $\log \Pi(|A-A^\ast |\leq C\epsilon)\asymp \log\epsilon$, the conclusions can be proved by applying part c and part d in Lemma \ref{le:conprob}.

\subsection{Proof of Lemma \ref{le:coen}}
(Proof of a) For any $\lambda\in \mathcal{F}^{\Lambda}_{s,\epsilon}$, let $S(\lambda)$ be the index set of the $s$ largest $\lambda_j$'s. For any $\lambda\in\mathcal{F}^{\Lambda}_{s,\epsilon}$, if $\lambda'\in \Lambda$ satisfies
$\lambda'_j=0$, for $j\in S^c(\lambda)$ and $|\lambda'_j-\lambda_j|\leq\epsilon/s$, for $j\in S(\lambda)$,
then $d_{\Sigma}(\lambda,\lambda')\leq\kappa||\lambda'-\lambda||_1\leq 2\kappa\epsilon$. Therefore, for a fixed index set $S\subset\{1,\ldots,M\}$ with size $s$, the set of all grid points in $[0,1]^s$ with mesh size $\epsilon/s$ forms an $2\kappa\epsilon$-covering set for all $\lambda$ such that $S(\lambda)=S$. Since there are at most ${M\choose s}$ such an $S$, the minimal $2\kappa\epsilon$-covering set for $\mathcal{F}^{\Lambda}_{s,\epsilon}$ has at most
${M\choose s}\times \big(\frac{s}{\epsilon}\big)^s$
elements, which implies that
\begin{align*}
    \log N(2\kappa\epsilon, \mathcal{F}^{\Lambda}_{s,\epsilon}, ||\cdot||_1)\leq \log {M\choose s}+s\log\frac{s}{\epsilon}\lesssim s\log \frac{M}{\epsilon}.
\end{align*}
This proves the first conclusion.

For any $\eta\in \mathcal{F}^{\eta}_{B,s,\epsilon}$, let $S(\eta)$ be the index set of the $s$ largest $|\eta_j|$'s. Similarly, for any $\lambda=A\eta\in\mathcal{F}^{\eta}_{B,s,\epsilon}$, if $\eta'\in D_{M-1}$ satisfies
$\eta'_j=0$, for $j\in S^c(\eta)$ and $|\eta'_j-\eta_j|\leq\epsilon/(Bs)$, for $j\in S(\eta)$, and $A'\leq B$ satisfies $|A'-A|\leq\epsilon$,
then $d_F(A'\eta',A\eta)\leq\kappa||A'\eta'-A\eta||_1\leq\kappa|A-A'|
+B\kappa||\eta'-\eta||_1\leq 3\kappa\epsilon$. Similar to the arguments for $\mathcal{F}^{\Lambda}_{s,\epsilon}$, we have
\begin{align*}
    \log N(3\kappa\epsilon, \mathcal{F}^{\eta}_{B,s,\epsilon}, ||\cdot||_1)\leq \log {M\choose s}+s\log\frac{Bs}{\epsilon}+\log\frac{B}{\epsilon}\lesssim s\log \frac{M}{\epsilon}+s\log B.
\end{align*}

(Proof of b)  By Lemma \ref{le:prep2}, any $\lambda\in \Lambda$ and $\eta\in BD_{M-1}$ can be approximated by an $m$-sparse vector in the same space with error $Cm^{-1/2}$ and $CBm^{-1/2}$ respectively. Moreover, by the proof of Lemma \ref{le:prep2}, all components of such $m$-sparse vectors are multiples of $1/m$. Therefore, a minimal $C/\sqrt{m}$-covering set of $\Lambda$ has at most
${M+m-1\choose m-1}$ elements, which is the total number of nonnegative integer solutions $(n_1,\ldots,n_M)$ of the equation: $n_1+\cdots+n_M=m$.
Therefore,
\begin{align*}
   & \log N(C/\sqrt{m}, \Lambda, d_{\Sigma})\leq\log {M+m-1\choose m-1} \lesssim m\log M,\\
   &\log N(CB/\sqrt{m}, \Lambda, d_{\Sigma})\leq\log {M+m-1\choose m-1} +\log\frac{B}{B/\sqrt{m}}\lesssim m\log M.
\end{align*}

\subsection{Proof of Lemma \ref{le:cpp}}
(Proof of a)  Consider a random probability $P$ drawn from the Dirichlet process (DP) $DP\big((\alpha/M^{\gamma-1})U\big)$ with concentration parameter $\alpha/M^{\gamma-1}$ and the uniform distribution $U$ on the unit interval $[0,1]$. Then by the relationship between the DP and the Dirichlet distribution, we have
\begin{align*}
    (\lambda_1,\ldots,\lambda_M)\sim \big(P(A_1),\ldots,P(A_M)\big),
\end{align*}
with $A_k=[(k-1)/M,k/M)$ for $k=1,\ldots,M$.  The stick-breaking representation for DP \citep{sethuraman1994} gives
$Q=\sum_{k=1}^{\infty}w_k\delta_{\xi_k}$, a.s.
where $\xi_k\overset{iid}{\sim}U$ and
\begin{align*}
    w_k=w_k'\prod_{i=1}^{k-1}(1-w_i'), \text{ with } w_k'\overset{iid}{\sim}\text{Beta}(1,\alpha/M^{\gamma-1}).
\end{align*}
For each $k$, let $i(k)$ be the unique index such that $\xi_k\in A_{i(k)}$. Let $\lambda_{(1)}\geq\cdots\geq\lambda_{(M)}$ be an ordering of $\lambda_1,\ldots,\lambda_M$, then
\begin{align*}
    \sum_{j=1}^s\lambda_{(j)}\geq
   Q\bigg(\bigcup_{j=1}^sA_{i(j)}\bigg)
= \sum_{k:\xi_k\in\bigcup_{j=1}^sA_{i(j)}}w_k
\geq \sum_{k=1}^s w_k.
\end{align*}
Combining the above with the definition of $w_k$ provides
\begin{align*}
    \sum_{j=s+1}^M\lambda_{(j)}\leq 1- \sum_{k=1}^s w_k'\prod_{i=1}^{k-1}(1-w_i')=\prod_{k=1}^s(1-w_k')\triangleq \prod_{k=1}^s v_k,
\end{align*}
where $v_k=1-w_k'\overset{iid}{\sim}\text{Beta}(\alpha/M^{\gamma-1},1)$. Since $v_k\in(0,1)$, we have
$(\mathcal{F}^{\Lambda}_{s,\epsilon})^c=\big\{\sum_{j=s+1}^M\lambda_{(j)}\geq\epsilon\big\}\subset
   \big\{\prod_{k=1}^s v_k\geq\epsilon\big\}$. Because
   \[
   Ev_k^s=\int_{0}^1\frac{\alpha}{M^{\gamma-1}}t^{\alpha/M^{\gamma-1}+s-1}dt=\frac{\alpha}{\alpha+M^{\gamma-1}s}\leq \alpha M^{-(\gamma-1)}s^{-1},
   \]
an application of Markov's inequality yields
\begin{align*}
    \Pi\bigg\{\prod_{k=1}^s v_k\geq\epsilon\bigg\}\leq \epsilon^{-s}\prod_{k=1}^sEv_k^s\lesssim M^{-s(\gamma-1)}s^{-s}\epsilon^{-s}.
\end{align*}
As a result,
\begin{align*}
    \Pi(\lambda\notin\mathcal{F}_{s,\epsilon}^{\lambda})\leq \Pi\bigg\{\prod_{k=1}^s v_k\geq\epsilon\bigg\}\leq \exp\bigg(-Cs(\gamma-1)\log \frac{M}{\epsilon}\bigg).
\end{align*}

(Proof of b) The proof is similar to that of a since $(|\eta_1|,\ldots,|\eta_M|)\sim (\lambda_1,\ldots,\lambda_M)$ and $\Pi(A>B)\leq e^{-CB}$ for $A\sim$ Ga$(a_0,b_0)$.

\subsection{Proof of Lemma \ref{le:test}}
(Proof of a) Under (A3), the conclusion can be proved by applying Lemma 2.1 and Lemma 4.1 in \cite{Kleijn2006} by noticing the fact that $||\sum_{j=1}^M\lambda_jf_j-f^\ast ||_Q=d_{\Sigma}(\lambda,\lambda^\ast )$.

\noindent (Proof of b) Let $\psi(\lambda,Y)=\frac{1}{2\sigma^2}||Y-F\lambda||_2^2$. We construct the test function as $\phi_n(Y)=I(\psi(\lambda^*,Y)-\psi(\lambda_2,Y)\geq 0)$. By the choice of $\lambda^\ast$, under $P_0$ we can decomposition the response $Y$ as $Y=F\lambda^\ast+\zeta+\epsilon$, where $\epsilon\sim N(0,\sigma^2I_n)$ and $\zeta=F_0-F\lambda^\ast\in\bbR^d$ satisfying $F^T\zeta=0$. By Markov's inequality, for any $t<0$, we have
\begin{align}
   P_{\lambda^\ast}\phi_n(Y)=&\ P_{\lambda^\ast}(e^{t\{\psi(\lambda^\ast,Y)-\psi(\lambda_2,Y)\}}\geq 1)\nonumber\\
   \leq &\ E_{\lambda^\ast}\exp\bigg\{\frac{t}{2\sigma^2}
   \big(||\zeta+\epsilon||_2^2-||F(\lambda^\ast-\lambda_2)+\zeta+\epsilon||_2^2\big)\bigg\}\nonumber\\
   = &\ E_{\lambda^\ast}\exp\bigg\{\frac{t}{\sigma^2}(\lambda_2-\lambda^\ast)^TF^T\epsilon\bigg\}
   \exp\bigg\{-\frac{t}{2\sigma^2} nd_{F}^2(\lambda_2,\lambda^\ast)\bigg\}\nonumber\\
   = &\ \exp\big\{-t (2\sigma^2)^{-1}nd_F^2(\lambda_2,\lambda^\ast)
   +t^2\sigma^{-2}nd^2_F(\lambda_2,\lambda^\ast)\big\},\nonumber\\
   = &\ \exp\big\{-(16\sigma^2)^{-1}nd^2_F(\lambda_2,\lambda^\ast)\big\},\label{eq:test1}
\end{align}
with $t=\frac{1}{4}>0$, where we have used the fact that $\epsilon\sim N(0,\sigma^2 I_n)$ under $P_{\lambda^\ast}$ and $F^T\zeta=0$. Similarly, for any $\lambda\in\bbR^M$, under $P_{\lambda}$ we have $Y=F\lambda+\epsilon$ with $\epsilon\sim N(0,\sigma^2I_n)$. Therefore, for any $t>0$ we have
\begin{align}
   P_{\lambda}(1-\phi_n(Y))=&\ P_{\lambda}(e^{t\{\psi(\lambda_2,Y)-\psi(\lambda^\ast,Y)\}}\geq 1)\nonumber\\
   \leq &\ E_{\lambda}\exp\bigg\{\frac{t}{2\sigma^2}
   \big(||\epsilon-F(\lambda_2-\lambda)||_2^2-||\epsilon-F(\lambda^\ast-\lambda)||_2^2\big)\bigg\}\nonumber\\
   = &\  E_{\lambda}\exp\bigg\{-\frac{t}{\sigma^2}(\lambda_2-\lambda^\ast)^TF^T\epsilon\bigg\}
   \exp\bigg\{-\frac{t}{2\sigma^2} n(d_{F}^2(\lambda,\lambda^\ast)-d_{F}^2(\lambda,\lambda_2))\bigg\}\nonumber\\
   = &\ \exp\big\{-t (2\sigma^2)^{-1}n\big(d_{F}^2(\lambda,\lambda^\ast)-d_{F}^2(\lambda,\lambda_2)\big)
   +t^2\sigma^{-2}nd^2_F(\lambda_2,\lambda^\ast)\big\},\nonumber\\
   = &\ \exp\bigg\{-(16\sigma^2)^{-1}n\frac{\big(d_F^2(\lambda,\lambda^\ast)-d_F^2(\lambda,\lambda_2)\big)^2}
   {d^2_F(\lambda_2,\lambda^\ast)}\bigg\},\label{eq:test2}
\end{align}
with $t=\frac{1}{4}\big(d_F^2(\lambda,\lambda^\ast)-d_F^2(\lambda,\lambda_2)\big)/d^2_F(\lambda_2,\lambda^\ast)>0$ if $d_F(\lambda,\lambda^\ast)>d_F(\lambda,\lambda_2)$.

Combining \eqref{eq:test1} and \eqref{eq:test2} yields
\begin{align*}
    &P_{\lambda^\ast}\phi_n(Y)\leq \ \exp\big\{-(16\sigma^2)^{-1}nd^2_F(\lambda_2,\lambda^\ast)\big\}\\
    \sup_{\lambda\in\bbR^M:\ d_F(\lambda,\lambda_2)<\frac{1}{4}d_F(\lambda^\ast,\lambda_2)}
    &P_{\lambda}(1-\phi_n(Y))\leq \ \exp\big\{-(64\sigma^2)^{-1}nd^2_F(\lambda_2,\lambda^\ast)\big\}.
\end{align*}

\section*{Acknowledgments}
This research was supported by grant R01ES020619 from the National Institute of Environmental Health Sciences (NIEHS) of the National Institutes of Health
(NIH).

\vskip 1em \centerline{\Large \bf APPENDIX} \vskip 1em
\setcounter{subsection}{0}
\renewcommand{\thesubsection}{A.\arabic{subsection}}
\setcounter{equation}{0}
\renewcommand{\theequation}{A.\arabic{equation}}

In the appendix, we provide details of the MCMC implementation for CA and LA. The key idea is to augment the weight vector $\lambda=(\lambda_1,\ldots,\lambda_M)\sim$ Diri$(\rho,\ldots,\rho)$ by
$\lambda_j=T_j/(T_1+\cdots+T_M)$ with $T_j \overset{iid}{\sim} Ga(\rho,1)$ for $j=1,\ldots,M$ and conduct Metropolis Hastings updating for $\log T_j$'s. Recall that $F=(F_j(X_i))$ is the $n\times M$ prediction matrix.

\subsection{Convex aggregation}
By augmenting the Dirichlet distribution in the prior for CA, we have the following Bayesian convex aggregation model:
\begin{align*}
    Y_i=&\sum_{j=1}^M \lambda_j F_{ij}+\epsilon_i, \ \epsilon_i\sim N(0,1/\phi),\\
    \lambda_j=&\frac{T_j}{T_1+\cdots+T_M},\
    T_j \sim Ga(\rho,1),\ \phi\sim Ga(a_0,b_0).
\end{align*}
We apply a block MCMC algorithm that iteratively sweeps through the following steps, where superscripts ``O", ``P" and ``N" stand for ``old", ``proposal" and ``new" respectively:
\begin{description}
  \item[1. Gibbs updating for $\phi$:] Updating $\phi$ by sampling from
  $[\phi|-] \sim Ga(a_n,b_n)$ with
  \begin{align*}
    a_n=a_0+\frac{n}{2},\quad b_n=b_0+\frac{1}{2}\sum_{i=1}^n\bigg(Y_i-\sum_{j=1}^M\lambda_jF_{ij}\bigg)^2.
  \end{align*}
  \item[2. MH updating for $T$($\lambda$):] For $j=1$ to $M$, propose $T_j^{P}=T_j^{O}e^{\beta U_j}$, where $U_j\sim U(-0.5,0.5)$. Calculate $\lambda^P_j=T_j^P/(\sum_{j=1}^MT^P_j)$ and the log acceptance ratio
      \begin{align*}
        \log R=&\frac{\phi}{2}\sum_{i=1}^n\bigg(Y_i-\sum_{j=1}^M\lambda^P_jF_{ij}\bigg)^2
        -\frac{\phi}{2}\sum_{i=1}^n\bigg(Y_i-\sum_{j=1}^M\lambda^O_jF_{ij}\bigg)^2\quad(\text{log-likelihood})\\
       & +\sum_{j=1}^M\big((\rho-1)\log T_j^P-T_j^P\big)-\sum_{j=1}^M\big((\rho-1)\log T_j^O-T_j^O\big)\quad(\text{log-prior})\\
       & + \sum_{j=1}^M \log T_j^P-\sum_{j=1}^M \log T_j^O \quad(\text{log-transition probability}).
      \end{align*}
      With probability $\min\{1, R\}$, set $T_j^N=T_j^P$, $j=1,\ldots,M$ and with probability $1-\min\{1, R\}$, set $T_j^N=T_j^O$, $j=1,\ldots,M$.
      Set  $\lambda^N_j=T_j^N/(\sum_{j=1}^MT^N_j)$, $j=1,\ldots,M$.
\end{description}
In the above algorithm, $\beta$ serves as a tuning parameter to make the acceptance rate of $T$ around $40\%$.

\subsection{Linear aggregation}
By augmenting the double Dirichlet distribution in the prior for LA, we have the following Bayesian linear aggregation model:
\begin{align*}
    Y_i=&\sum_{j=1}^M \theta_j F_{ij}+\epsilon_i, \ \epsilon_i\sim N(0,1/\phi),\ \theta_j=Az_j\lambda_j,\
    \lambda_j=\frac{T_j}{T_1+\cdots+T_M},\\
    A \sim& Ga(c_0,d_0),\ z_j\sim Bernoulli(0.5),\ T_j \sim Ga(\rho,1),\ \phi\sim Ga(a_0,b_0).
\end{align*}

The MCMC updating of $T$ (or equivalently $\lambda$) and $\phi$ is the similar as those in the convex aggregation. In each iteration of the block MCMC algorithm, we add two additional steps for updating $z$ and $A$:

\begin{description}
  \item[3. MH updating for $A$:] Propose $A^{P}=A^{O}e^{\beta U}$, where $U_j\sim U(-0.5,0.5)$. Calculate $\lambda^P_j=\lambda^O_je^{\beta U}$ and the log acceptance ratio
      \begin{align*}
        \log R=&\frac{\phi}{2}\sum_{i=1}^n\bigg(Y_i-\sum_{j=1}^M\lambda^P_jF_{ij}\bigg)^2
        -\frac{\phi}{2}\sum_{i=1}^n\bigg(Y_i-\sum_{j=1}^M\lambda^O_jF_{ij}\bigg)^2\quad(\text{log-likelihood})\\
       & +\big((c_1-1)\log A^P-d_1 A^P\big)-\big((c_1-1)\log A^O-d_1 A^O\big)\quad(\text{log-prior})\\
       & + \log A^P- \log A^O \quad(\text{log-transition probability}).
      \end{align*}
      With probability $\min\{1, R\}$, set $A^N=A^P$ and with probability $1-\min\{1, R\}$, set $A^N=A^O$.
      Set  $\lambda^N_j=\lambda^O_j A^N/A^O$, $j=1,\ldots,M$.
  \item[4. MH updating for $z$:] For $j=1$ to $M$, propose $z_j^{P}=z_j^{O}V_j$, where $P(V_j=\pm 1)=0.5$. Calculate $\lambda^P_j=\lambda^O_j V_j$ and the log acceptance ratio
      \begin{align*}
        \log R=&\frac{\phi}{2}\sum_{i=1}^n\bigg(Y_i-\sum_{j=1}^M\lambda^P_jF_{ij}\bigg)^2
        -\frac{\phi}{2}\sum_{i=1}^n\bigg(Y_i-\sum_{j=1}^M\lambda^O_jF_{ij}\bigg)^2\quad(\text{log-likelihood}).
      \end{align*}
      With probability $\min\{1, R\}$, set $z_j^N=z_j^P$, $j=1,\ldots,M$ and with probability $1-\min\{1, R\}$, set $z_j^N=z_j^O$, $j=1,\ldots,M$.
      Set  $\lambda^N_j=\lambda^O_jz_j^{P}/z_j^O$, $j=1,\ldots,M$.
\end{description}

\bibliography{draft}

\begin{thebibliography}{}

\bibitem[\protect\citeauthoryear{Bhattacharya, Pati, Pillai, and
  Dunson}{Bhattacharya et~al.}{2013}]{Anirban2013}
Bhattacharya, A., D.~Pati, N.~S. Pillai, and D.~B. Dunson (2013).
\newblock Bayesian shrinkage.
\newblock {\em arXiv:1212.6088\/}.

\bibitem[\protect\citeauthoryear{Birg\'{e}}{Birg\'{e}}{2004}]{Birge2004}
Birg\'{e}, L. (2004).
\newblock Model selection for {G}uassian regression with random design.
\newblock {\em Bernoulli\/}~{\em 10}, 1039--1151.

\bibitem[\protect\citeauthoryear{Blei, Ng, and Jordan}{Blei
  et~al.}{2003}]{Blei2003}
Blei, D.~M., A.~Y. Ng, and M.~I. Jordan (2003).
\newblock Latent {D}irichlet allocation.
\newblock {\em Journal of Machine Learning Research\/}~{\em 3}, 993--1022.

\bibitem[\protect\citeauthoryear{B\"{u}hlmann}{B\"{u}hlmann}{2006}]{Peter2006}
B\"{u}hlmann, P. (2006).
\newblock Boosting for high-dimensional linear models.
\newblock {\em Ann. Statist.\/}~{\em 34}, 559--583.

\bibitem[\protect\citeauthoryear{Bunea and Nobel}{Bunea and
  Nobel}{2008}]{Bunea2008}
Bunea, F. and A.~Nobel (2008).
\newblock Sequential procedures for aggregating arbitrary estimators of a
  conditional mean.
\newblock {\em IEEE Transactions on Information Theory\/}~{\em 54}, 1725--1735.

\bibitem[\protect\citeauthoryear{Bunea and Tsybakov}{Bunea and
  Tsybakov}{2007}]{Bunea2007}
Bunea, F. and Tsybakov (2007).
\newblock Aggregation for {G}aussian regression.
\newblock {\em Ann. Statist.\/}~{\em 35}, 1674--1697.

\bibitem[\protect\citeauthoryear{Carvalho, Polson, and Scott}{Carvalho
  et~al.}{2010}]{Carvalho2010}
Carvalho, C.~M., N.~G. Polson, and J.~G. Scott (2010).
\newblock The horseshoe estimator for sparse signals.
\newblock {\em Biometrika\/}~{\em 97}, 465--480.

\bibitem[\protect\citeauthoryear{Castillo and van~der Vaart}{Castillo and
  van~der Vaart}{2012}]{Castillo2012}
Castillo, I. and A.~W. van~der Vaart (2012).
\newblock Needles and straw in a haystack: {P}osterior concentration for
  possibly sparse sequences.
\newblock {\em Ann. Statist.\/}~{\em 40}, 2069--2101.

\bibitem[\protect\citeauthoryear{de~Jonge and van Zanten}{de~Jonge and van
  Zanten}{2013}]{Jonge2013}
de~Jonge, R. and H.~van Zanten (2013).
\newblock Semiparametric bernstein¨cvon mises for the error standard deviation.
\newblock {\em Electronic Journal of Statistics\/}~{\em 7}, 217--243.

\bibitem[\protect\citeauthoryear{George and McCulloch}{George and
  McCulloch}{1997}]{George1997}
George, E.~I. and R.~E. McCulloch (1997).
\newblock Approaches for {B}ayesain variable selection.
\newblock {\em Statistica Sinica\/}~{\em 7}, 339--373.

\bibitem[\protect\citeauthoryear{Ghosal, Ghosh, and Van Der~Vaart}{Ghosal
  et~al.}{2000}]{Ghosal2000}
Ghosal, S., J.~K. Ghosh, and A.~W. Van Der~Vaart (2000).
\newblock Convergence rates of posterior distributions.
\newblock {\em Ann. Statist.\/}~{\em 28}, 500--531.

\bibitem[\protect\citeauthoryear{Ghosal and van~der Vaart}{Ghosal and van~der
  Vaart}{2007}]{Ghosal2007}
Ghosal, S. and A.~W. van~der Vaart (2007).
\newblock Convergence rates of posterior distributions for noniid observations.
\newblock {\em Ann. Statist.\/}~{\em 35}, 192--223.

\bibitem[\protect\citeauthoryear{Guedj and Alquier}{Guedj and
  Alquier}{2013}]{Guedj2013}
Guedj, B. and P.~Alquier (2013).
\newblock {PAC}-{B}ayesian estimation and prediction in sparse additive models.
\newblock {\em Electronic Journal of Statistics\/}~{\em 7}, 264--291.

\bibitem[\protect\citeauthoryear{Hoerl and Kennard}{Hoerl and
  Kennard}{1970}]{Hoerl1970}
Hoerl, A. and R.~Kennard (1970).
\newblock Ridge regression: Biased estimation for nonorthogonal problems.
\newblock {\em Technometrics\/}~{\em 12}, 69--82.

\bibitem[\protect\citeauthoryear{Hoeting, Madigan, Raftery, and
  Volinsky}{Hoeting et~al.}{1999}]{Hoeting1999}
Hoeting, J.~A., D.~Madigan, A.~E. Raftery, and C.~T. Volinsky (1999).
\newblock Bayesian model averaging: {A} tutorial.
\newblock {\em Statistical Science\/}~{\em 14}, 382--417.

\bibitem[\protect\citeauthoryear{Ishwaran and Rao}{Ishwaran and
  Rao}{2005}]{Ishwaran2005}
Ishwaran, H. and J.~S. Rao (2005).
\newblock Spike and slab variable selection: {F}requentist and {B}ayesian
  strategies.
\newblock {\em Ann. Statist.\/}~{\em 33}, 730--773.

\bibitem[\protect\citeauthoryear{Juditsky and Nemirovski}{Juditsky and
  Nemirovski}{2000}]{Juditsky2000}
Juditsky, A. and A.~Nemirovski (2000).
\newblock Functional aggregation for nonparametric regression.
\newblock {\em Ann. Statist.\/}~{\em 28}, 681--712.

\bibitem[\protect\citeauthoryear{Kleijn and van~der Vaart}{Kleijn and van~der
  Vaart}{2006}]{Kleijn2006}
Kleijn, B. J.~K. and A.~W. van~der Vaart (2006).
\newblock Misspecification in infinite-dimensional {B}ayesian statistics.
\newblock {\em Ann. Statist.\/}~{\em 34}, 837--877.

\bibitem[\protect\citeauthoryear{Nemirovski}{Nemirovski}{2000}]{Nemirovski2000}
Nemirovski, A. (2000).
\newblock {\em Topics in non-parametric statistics.}
\newblock Lectures on Probability Theory and Statistics (Saint-Flour, 1998).
  ecture Notes in Math. 1738. Springer, Berlin.

\bibitem[\protect\citeauthoryear{Polson and Scott}{Polson and
  Scott}{2010}]{Polson2010}
Polson, N.~G. and J.~G. Scott (2010).
\newblock {\em Shrink globally, act locally: {S}parse {B}ayesian regularization
  aand prediction}.
\newblock Bayesian Statistics 9 (J.M. Bernardo, M.J. Bayarri, J.O. Berger, A.P.
  Dawid, D. Heckerman, A.F.M. Smith and M. West, eds.). Oxford University
  Press, New York.

\bibitem[\protect\citeauthoryear{Raskutti, Wainwright, and Yu}{Raskutti
  et~al.}{2011}]{Raskutti2011}
Raskutti, G., M.~Wainwright, and B.~Yu (2011).
\newblock Minimax rates of estimation for high-dimensional linear regression
  over $l_q$-balls.
\newblock {\em IEEE transactions on information theory\/}~{\em 57}, 6976--6994.

\bibitem[\protect\citeauthoryear{Rousseau and Mengersen}{Rousseau and
  Mengersen}{2011}]{Rousseau2011}
Rousseau, J. and K.~Mengersen (2011).
\newblock Asymptotic behaviour of the posterior distribution in overfitted
  mixture models.
\newblock {\em J. R. Statist. Soc. B\/}~{\em 73}, 689--710.

\bibitem[\protect\citeauthoryear{Sethuraman}{Sethuraman}{1994}]{sethuraman1994}
Sethuraman, J. (1994).
\newblock A constructive definition of {D}irichlet priors.
\newblock {\em Statistica Sinica\/}~{\em 4}, 639--650.

\bibitem[\protect\citeauthoryear{Tibshirani}{Tibshirani}{1996}]{Tibshirani1996}
Tibshirani, R. (1996).
\newblock Regression shrinkage and selection via the lasso.
\newblock {\em Journal of the Royal Statistical Association\/}~{\em 58},
  267--288.

\bibitem[\protect\citeauthoryear{Tsybakov}{Tsybakov}{2003}]{Tsybakov2003}
Tsybakov, A. (2003).
\newblock Optimal rates of aggregation.
\newblock {\em Learning Theory and Kernel Machines, Lecture Notes in Computer
  Science,\/}~{\em 2777}, 303--313.

\bibitem[\protect\citeauthoryear{van~der Laan, Polley, and Hubbard}{van~der
  Laan et~al.}{2007}]{vanderlann2007}
van~der Laan, M.~J., E.~C. Polley, and A.~E. Hubbard (2007).
\newblock Super learner.
\newblock {\em Genetics and Molecular Biology\/}~{\em 6}.

\bibitem[\protect\citeauthoryear{Wegkamp}{Wegkamp}{2003}]{Wegkamp2003}
Wegkamp, M. (2003).
\newblock Model selection in nonparametric regression.
\newblock {\em Ann. Statist.\/}~{\em 31}, 252--273.

\bibitem[\protect\citeauthoryear{Yang}{Yang}{2000}]{Yuhong2000}
Yang, Y. (2000).
\newblock Combining different procedures for adaptive regression.
\newblock {\em J. Multivariate Anal.\/}~{\em 74}, 135--161.

\bibitem[\protect\citeauthoryear{Yang}{Yang}{2001}]{Yuhong2001}
Yang, Y. (2001).
\newblock Adaptive regression by mixing.
\newblock {\em J. Amer. Statist. Assoc.\/}~{\em 96}, 574--588.

\bibitem[\protect\citeauthoryear{Yang}{Yang}{2004}]{Yuhong2004}
Yang, Y. (2004).
\newblock Aggregating regression procedures to improve performance.
\newblock {\em Bernoulli\/}~{\em 10}, 25--47.

\end{thebibliography}
\end{document}